\documentclass{article}

\usepackage[utf8]{inputenc}

\usepackage{latexsym}
\usepackage{amssymb}
\usepackage{amsmath}
\usepackage{amsthm}
\usepackage{amsfonts}

\usepackage{graphicx}
\usepackage[english]{babel}

\usepackage{fontenc}

\usepackage{enumerate}
\usepackage{textcomp}
\usepackage{geometry}
\usepackage{mathrsfs}
\usepackage{natbib}
\usepackage{fancyhdr}
\usepackage[colorlinks=true,linkcolor=blue,citecolor=blue]{hyperref}
\usepackage{natbib}
\usepackage{graphicx}
\usepackage[all]{xy}

\usepackage{stmaryrd}

\newcommand{\G}{{\mathbb{G}}}
\newcommand{\B}{{\mathbb{B}}}
\newcommand{\K}{{\mathbb{K}}}
\newcommand{\E}{{\mathbb{E}}}
\newcommand{\J}{{\mathbb{J}}}
\newcommand{\R}{{\mathbb{R}}}
\newcommand{\M}{{\mathbb{M}}}
\newcommand{\U}{{\mathbb{U}}}
\newcommand{\V}{{\mathbb{V}}}

\newcommand{\cP}{{\mathcal{P}}}
\newcommand{\cL}{{\mathcal{L}}}
\newcommand{\cR}{{\mathcal{R}}}

\newcommand{\cV}{{\mathcal{V}}}
\newcommand{\cA}{{\mathcal{A}}}
\newcommand{\cC}{{\mathcal{C}}}
\newcommand{\cD}{{\mathcal{D}}}
\newcommand{\cS}{{\mathcal{S}}}
\newcommand{\cW}{{\mathcal{W}}}

\newcommand{\cQ}{{\mathcal{Q}}}
\newcommand{\cH}{{\mathcal{H}}}
\newcommand{\cF}{{\mathcal{F}}}
\newcommand{\cE}{{\mathcal{E}}}

\newcommand{\X}{{\mathcal{X}}}

\newcommand{\I}{{\mathrm{I}}}
\newcommand{\Exp}{{\mathrm{Exp}}}
\newcommand{\Log}{{\mathrm{Log}}}
\newcommand{\St}{{\mathrm{St}}}
\newcommand{\tr}{{\mathrm{tr}}}

\newcommand{\rk}{{\mathrm{rank}}}

\newcommand{\Ct}{{\mathrm{Cut}}}
\newcommand{\Img}{{\mathrm{Im}}}

\newcommand{\Aut}{{\mathrm{Aut}}}

\newcommand{\Ad}{{\mathrm{Ad}}}
\newcommand{\ad}{{\mathrm{ad}}}
\newcommand{\inj}{{\mathrm{inj}}}
\newcommand{\Id}{{\mathrm{Id}}}

\newcommand{\g}{{\mathfrak{g}}}
\newcommand{\F}{{\mathcal{F}^{\mathrm{I}}}}
\newcommand{\Gr}{{G_{q}}}

\newcommand{\Gras}{{\mathrm{Gr}}}

\newcommand{\so}{{\mathfrak{so}(n)}}

\newcommand{\fm}{{\mathfrak{m}}}
\newcommand{\fk}{{\mathfrak{k}}}

\newcommand{\fv}{{\mathfrak{v}}}
\newcommand{\fS}{{\mathfrak{S}}}

\newcommand{\lb}{{\llbracket}}
\newcommand{\rb}{{\rrbracket}}

\newcommand{\pii}{{\pi^{\mathrm{I}}}}
\newcommand{\piio}{{\pi_{0}^{\mathrm{I}}}}

\newcommand{\piz}{{\pi_{0}}}

\newcommand{\son}{{\mathfrak{so}(n)}}

\newcommand{\Sy}{{\mathrm{Sym}}}

\newcommand{\lra}{{\longrightarrow}}
\newcommand{\lmp}{{~\longmapsto~}}

\newcommand\at[2]{\left.#1\right|_{#2}}

\newtheorem{rem}{Remark}
\newtheorem{prop}{Proposition}
\newtheorem{corr}{Corollary}
\newtheorem{lem}{Lemma}

\newtheorem{theo}{Theorem}
\newtheorem{defi}{Definition}

\newtheorem{Assum}{Assumption}

\begin{document}

\title{Effective formulas for the geometry of normal homogeneous spaces. Application to flag manifolds.}

\author{Dimbihery Rabenoro and Xavier Pennec}



\maketitle

\begin{abstract}
Consider a smooth manifold and an action on it of a compact connected Lie group with a bi-invariant metric. Then, any orbit is an embedded submanifold that is isometric to a normal homogeneous space for the group. In this paper, we establish new explicit and intrinsic formulas for the geometry of any such orbit. We derive our formula of the Levi-Civita connection from an existing generic formula for normal homogeneous spaces, i.e. which determines, a priori only theoretically, the connection. We say that our formulas are effective: they are directly usable, notably in numerical analysis, provided that the ambient manifold is convenient for computations. Then, we deduce new effective formulas for flag manifolds, since we prove that they are orbits under a suitable action of the special orthogonal group on a product of Grassmannians. This representation of them is quite useful, notably for studying flags of eigenspaces of symmetric matrices. Thus, we develop from it a geometric solution to the problem of perturbation of eigenvectors which is explicit and optimal in a certain sense, improving the classical analytic solution. 
\end{abstract}

\section{Introduction}\label{sect1}

A flag of $\R^{n}$ is a filtration $\mathcal{V}$ of linear subspaces $ V_0 = \{ 0 \} \subset V_1 \subset \ldots \subset V_{r} = \R^{n} $, where $r \geq 2$. For all $1 \leq i \leq r$, let $W_{i}$ be the orthogonal complement of $V_{i-1}$ in $V_{i}$. Then, the sequence $(W_{i})_{1 \leq i \leq r}$ is an equivalent characterization of the flag $\mathcal{V}$. In this paper,  we adopt this representation of flags. Thus, a flag of $\R^{n}$ is a sequence $\mathcal{W}=(W_{i})_{1 \leq i \leq r}$ of mutually orthogonal linear subspaces that spans $\R^{n}$. The sequence $\I=(q_{i})_{1 \leq i \leq r}$, where $q_{i} := \dim(W^{i})$, is called the \textit{type} of the flag $\mathcal{W}$. Thus, for $r=2$, we recover Grassmannians. It is well-known that the set of all flags of a given type has a structure of \textit{normal} homogeneous space for $SO(n)$, which we call a flag manifold.

\subsection{Context and main result}

Such spaces have recently been a rising subject of interest in statistical machine learning for generalizing Grassmannians and handling multiple subspaces at the same time, as in \cite{Nishimori Akaho and Plumbley 2006}. The geometric structure provides more intrinsic solutions and wider perspectives to issues from applications. Thus, it was shown in \cite{Pennec 2018} that PCA can be rephrased as an optimization on a flag manifold. Namely, one of the novelty of \cite{Pennec 2018} is to realize that the flag resulting from Principal Component Analysis (PCA) actually optimizes a criterion on the flag manifold: the accumulated unexplained variance (AUV). On this basis, new robust variants like the minimal accumulated unexplained p-variance have been proposed, even for manifold-valued data \cite{Pennec 2017}. The ubiquitous appearance of flags in applications, exemplified in \cite{Ye Wong and Lim 2022}, led these authors to develop tools for Riemannian optimization on flag manifolds. However, formulas of the Levi-Civita connection and the Riemannian curvature are not given therein. 

In the particular case of $\Gras(n, q)$, a formula for the connection is available in \cite{Bendokat Zimmermann and Absil 2020}, where $\Gras(n, q)$ is identified to orthogonal projectors. Nevertheless, this formula is not explicit, since it requires an extension of vector fields to the space $\Sy_{n}$ of symmetric matrices. An explicit formula is obtained in \cite{Absil Mahony and Sepulchre 2004}, but the representation of $\Gras(n, q)$ therein is less convenient than that of \cite{Bendokat Zimmermann and Absil 2020} which allows to express explicitly many Riemannian operations. 

Thus, one of our motivation was to obtain an explicit formula for the Levi-Civita connection of flag manifolds, in a representation of them that is convenient for computations. We say that such a formula is \textit{effective}. In this paper, we establish notably an explicit and intrinsic formula for the Levi-Civita connection of any normal homogeneous space: see Theorem $\ref{finalLCCapplied}$ in Section $\ref{sect3}$. Then, we deduce such formulas for flag manifolds, in a representation by orthogonal projectors which extends that of \cite{Bendokat Zimmermann and Absil 2020} for $\Gras(n, q)$: see Theorem $\ref{StatementTheoLCCflag}$ in Section $\ref{sect4}$.

\subsection{Effective formulas}

In practice, one can compare different representations of a given homogeneous space, in function of their efficiency for computations. In a general framework, we precise hereafter what we mean by effective formulas. Consider a manifold $\M$ and a smooth action on $\M$ of a compact connected Lie group $\G$ with a bi-invariant metric. For any orbit $\B$, fix a point $b_{0}$ in $\B$, called its origin. The \textit{main orbital map} is the surjective map $\piz$ defined by 
\begin{equation*}
\piz : \G \lra \B, \quad \piz(Q)=Q \cdot b_{0}
\end{equation*}

\noindent
Let $\K \subset \G$ be the isotropy group of $b_{0}$ and $\psi : \G \lra \G/\K$ the canonical quotient map. Then, by \cite[Proposition A.2]{Bendokat Zimmermann and Absil 2020}, $\B$ is an embedded submanifold of $\M$ and $\piz$ induces a diffeomorphism $\widehat{\piz} : \G/\K \lra \B$ such that $\piz=\widehat{\piz} \circ \psi$. Endowed with the quotient metric, 
$\G / \K$ is a \textit{normal homogeneous space}. Then, we endow $\B$ with the metric for which $\widehat{\piz}$ is an isometry. Thus, $\B$ is called a \textit{realization} of $\G / \K$ embedded in $\M$. We say that \textit{explicit} formulas for the geometry of any such orbit $\B$ are \textit{effective} in the sense that they are directly usable, notably in numerical analysis, provided that the ambient manifold $\M$ is convenient for computations. 

In this paper, we establish such effective formulas for the Levi-Civita connection and the sectional curvature of $\B$. These formulas simplify when $\B$ is additionally a \textit{symmetric space}. Contrarily to the Riemannian curvature, the connection is not a tensor, so that an explicit expression is much more difficult to obtain. Thus, we focus below on our method for the connection, based on a \textit{generic} formula for the Levi-Civita connection of normal homogeneous spaces, provided in \cite{Alekseevsky 1996} or in \cite{Arvanitoyeorgos 2003}. This formula is generic in the sense that it holds only for particular vector fields, the \textit{fundamental} ones, and it is implicit in \cite{Alekseevsky 1996} and \cite{Arvanitoyeorgos 2003} that it fully determines the connection. However, explicit formulas for arbitrary vector fields are not developed therein, nor elsewhere in the literature. 

Furthermore, in the numerical analysis literature, explicit formulas for the connection are derived in particular cases of such orbits, by specific developments on a case-by-case basis. Usually, the connection is obtained by associating to vector fields on such an orbit, other ones that are valued in a vector space, in which the connection is computed through usual derivation. Our method for any such orbit $\B$ is also based on this principle but we associate to any vector field on $\B$, a function valued in the Lie algebra $\g$ of $\G$. This lifting to $\g$ allows to exploit the rich structure of $\g$. Thus, the \textit{adjoint representation} plays an important role in our method.

\subsection{Our method for the connection}

By construction, the map $\piz : \G \lra \B$ is a \textit{Riemannian submersion}, which is essential in our method. Thus, in $\g$, we fix an orthonormal basis $(\epsilon_{k})_{k}$ of the horizontal space wrt $\piz$. We obtain readily that any vector field on $\B$ is decomposed into a linear combination of fundamental vector fields $(\epsilon_{k}^{*})_{k}$, associated to the $(\epsilon_{k})_{k}$. Then, we demonstrate that the coefficients of this decomposition are \textit{smooth functions} on a neighborhood of $b_{0}$ in $\B$. Thus, we establish a handleable expression for these coefficients, from which we prove their smoothness. The key point is the existence of a \textit{smooth} local section of $\piz : \G \lra \B$ which is induced by some diffeomorphisms between fibers under $\piz$, called \textit{holonomy maps}, built from the Riemannian submersion structure of $\piz$. 

From such a decomposition with \textit{smooth coefficients}, we deduce an explicit formula for arbitrary vector fields on $\B$ from the generic one, by applying the Leibniz rule for affine connections. Then, some factorizations allow us to establish a final explicit formula that is also \textit{intrinsic}, i.e. independent of the choice of the basis $(\epsilon_{k})_{k}$. Notice that a method using only O'Neill's formula \cite{O'Neill 1966} for the Riemannian submersion $\piz$ would not provide  such an explicit and intrinsic formula, for which we have fully exploited the structure of the Lie algebra $\g$.

\subsection{Flag manifolds as normal homogeneous spaces}\label{IntroAppliFlags}

By results of \cite{Nguyen 2022}, the flag manifold of a given type $\I=(q_{i})_{1 \leq i \leq r}$ is identified with an embedded submanifold of $\Sy_{n}$ that is diffeomorphic to 
$SO(n) / SO(\I)$, where 
\begin{equation}\label{OofI}
SO(\I) = SO(n) \cap O(\mathrm{I}) \quad \textrm{and} \quad O(\mathrm{I}):=\prod O(q_{i}). 
\end{equation}

\noindent
We endow $SO(n)$ with the metric $g^{O}$ induced by the Froebenius metric on matrices. Let $\son$ be the Lie algebra of $SO(n)$, i.e. the set of $n \times n$ skew-symmetric matrices. Then, this metric writes:
\begin{equation}\label{metricO}
g^{O}_{Q}\left(Q\Omega, Q\widetilde{\Omega} \right):=\frac{1}{2} \mathrm{tr} \left( \left(Q\Omega \right)^{T} Q\widetilde{\Omega} \right) = \frac{1}{2} \mathrm{tr} \left( \Omega^{T}\widetilde{\Omega} \right), 
\quad Q \in O(n), ~ \Omega, \widetilde{\Omega} \in \son. 
\end{equation}  

\noindent
Now, $g^{O}$ is a \textit{bi-invariant} metric and a negative multiple of the Killing form on $\son$. Thus, endowed with the quotient metric, a flag manifold is a \textit{standard normal homogeneous space}. Grassmannians are in addition \textit{symmetric spaces}, while all other flag manifolds are not even locally symmetric. 

In any case, our preceding results for normal homogeneous spaces should provide explicit formulas for the geometry of flag manifolds. However, instead of this embedding into $\Sy_{n}$, we consider hereafter a realization of $SO(n) / SO(\I)$ that is much more convenient for computations.

\subsection{Effective formulas for flag manifolds}

For $1 \leq q \leq n$, any linear subspace of dimension $q$ of $\R^{n}$ is identified with the orthogonal projector onto its range. Thus, the Grassmannian $\Gras(n, q)$ is identified with the following set
\begin{equation}\label{GrassProj}
\Gr = \left\{ P \in \R^{n \times n} : ~P^{2}=P ~;~ P^{T}=P ~;~ \rk(P)=q \right\} \subset \Sy_{n}.
\end{equation}

\noindent
By \cite{Bendokat Zimmermann and Absil 2020}, $\Gr$ is a realization of a quotient of $SO(n)$ embedded in $\Sy_{n}$. We generalize this representation to flag manifolds. Given a type $\I = (q_{i})_{1 \leq i \leq r}$ with $r \geq 3$, consider the product manifold 
\begin{equation}\label{flagProj}
G^{\I} := \prod G_{q_{i}}.
\end{equation}

\noindent
Now, the incremental subspaces of a flag are mutually orthogonal. Thus, the flag manifold of type $\I$ is identified with the submanifold $\F$ of $G^{\I}$ defined below, which is considered in \cite{Ye Wong and Lim 2022}: 
\begin{equation}\label{flagProj}
\F := \left\{ \mathcal{P}=(P_{i})_{1 \leq i \leq r} \in G^{\I} : ~P_{i}P_{j}=0, ~i \neq j  \right\}.
\end{equation}

\noindent
Now, consider the action of $SO(n)$ on $G^{\I}$ defined by 
\begin{equation}\label{actionOfOrtho}
(Q, \cP) \lmp Q \cdot \cP := \left( QP_{1}Q^{T}, ..., QP_{r}Q^{T}\right), \quad Q \in SO(n), \cP = (P_{i})_{i} \in G^{\I}.  
\end{equation}

\noindent
Let $\cP_{0} := \left( P_{0}^{i}\right)_{1 \leq i \leq r}$ be the \textit{standard flag}, where $P_{0}^{i}$ is the following block diagonal matrix:
\begin{equation}\label{DefStandardFlag}
P_{0}^{i}:=\mathrm{Diag}\left[ 0_{q_{1}}, ..., I_{q_{i}}, ..., 0_{q_{r}} \right],
\quad 1 \leq i \leq r.
\end{equation}

\noindent
This action appears in many domains involving flags, but without being formalized: see \cite{Anderson 1963} in statistics or \cite{Kato 1995} in perturbation theory. In \cite{Ye Wong and Lim 2022}, $\F$ is endowed with the restriction to $\F$ of the product metric $\sum g^{i}$ on $G^{\I}$, where for all $1 \leq i \leq r$, $g^{i}$ is the metric on $G_{q_{i}}$ defined in \cite{Bendokat Zimmermann and Absil 2020}. 

Now, we prove in paragraph $\ref{mainOMflag}$ that $\F$ is the orbit of $\cP_{0}$, whose isotropy group is $SO(\I)$ defined in $(\ref{OofI})$.  Thus, we introduce a \textit{new metric} on $\F$, for which $\F$ is a realization of $SO(n) / SO(\I)$ embedded in $G^{\I}$. Indeed, this metric does not coincide with that of \cite{Ye Wong and Lim 2022}: see Remark $\ref{metricDifferYeKe}$ in Section $\ref{sect4}$. Therefore, we deduce new explicit formulas for the Levi-Civita connection and the sectional curvature of $\F$ and $\Gr$ which are convenient for computations and can be implemented.

\subsection{Application to perturbation theory}

In $\F$, the main orbital map $\piio : SO(n) \lra \F$ is defined by $Q \lmp Q \cdot \cP_{0}$. Only the existence of a smooth local section of $\piio$ is required in our method for the connection. However, we wish to derive an explicit form of such a section. This section is based on holonomy maps, which are in turn built from the Logarithm map on $\F$. Now, the latter is not available in closed form on $\F$, but is on $\Gr$ : see \cite{Batzies Huper Machado and Silva Leite 2015}. In fact, the embedding of $\F$ into $G^{\I}$ allows us to derive an explicit expression for a section of the orbital map $\pii : Q \in O(n) \lmp Q \cdot \cP_{0}$, whose restriction to $SO(n)$ is $\piio$. 

Finally, we develop, from this section of $\pii$, a geometric solution to the problem of perturbation of eigenvectors. Indeed, the eigenspaces of a symmetric matrix form a flag. Our solution is \textit{explicit} and \textit{optimal} in a certain sense, improving the classical analytic one presented in \cite{Kato 1995}.

\subsection{Outline of the paper}

In Section $\ref{sect2}$, following \cite{Arvanitoyeorgos 2003}, we summarize results on the geometry of naturally reductive spaces, a class of homogeneous spaces including normal and symmetric ones. Thus, we recall notably the generic formulas for the Levi-Civita connection and the sectional curvature of normal homogeneous spaces. Then, in Section $\ref{sect3}$, we develop our method to derive the effective formulas for realizations of such spaces. Finally, in Section $\ref{sect4}$, we study the geometry of flag manifolds by applying the results of Section $\ref{sect3}$ and treat geometrically the question of perturbation of eigenvectors.

\section{Survey on homogeneous spaces}\label{sect2}

In this section, we survey the hierarchy of classes of homogeneous spaces aforementioned. We also summarize in tables the formulas for their geometry, which illustrate how these formulas follow from those for naturally reductive spaces. In particular, we locate precisely flag manifolds and Grassmannians within this hierarchy. These results are used in the following Sections.

The set of all smooth vector fields on a smooth manifold $\M$ is denoted by $\X(\M)$. For $X, Y \in \X(\M)$, their Lie bracket is denoted by $\left[ X, Y \right]^{\M}$ as a vector field. Throughout this section, $\G$ is a \textit{connected} Lie group with neutral element $e$ and Lie algebra $\g$. $\G$ is endowed with a metric $g^{\G}$ whose restriction to $\g$ is denoted by $\left\langle \cdot, \cdot \right\rangle := g^{\G}_{e}$.

\subsection{Preliminaries}

For $u, v \in \g$, their Lie bracket in $\g$ is denoted by $\lb u, v \rb$, where the double brackets allow to distinguish from the Lie bracket of vector fields on $\G$, denoted by $\left[ \cdot, \cdot \right]^{\G}$. Then, by definition, 
\begin{equation}\label{LieBraG}
\lb u, v \rb = \left[ L(u), L(v) \right]^{\G}_{e},
\end{equation} 

\noindent
where $L(u)$ is the left invariant vector field on $\G$ generated by $u$. Thus, if $\G$ is a matrix Lie group, then the Lie bracket in $\g$ is the commutator of matrices. Let $\K$ be a closed subgroup of $\G$ and $\psi : \G \lra \G/\K$ the canonical quotient map. The element $o:=\psi(e) = e\K$ of $\G / \K$ is called its \textit{origin}. Let $\fk$ be the Lie algebra of $\K$. Then, 
\begin{equation}\label{kerTangentPi}
\fk = \ker(d_{e}\psi). 
\end{equation}

\noindent
The \textit{orbital map} at $Q\K \in \G/\K$ is the map $\psi_{Q\K} : \G \lra \G/\K$ defined by
\begin{equation}\label{orbitalMapGK}
\psi_{Q\K}(Q') = Q' \cdot (Q\K) = (Q'Q)\K, \quad Q' \in \G. 
\end{equation}

\noindent
Thus, $\psi_{o}=\psi$. For $u \in \g$, the \textit{fundamental vector field} $\widetilde{u} \in \X(\M)$ is defined by 
\begin{equation*}
\widetilde{u}(Q\K) = \left(d_{e} \psi_{Q\K}\right)(u), \quad Q\K \in \G/\K. 
\end{equation*}

\noindent
Let $J_{Q} : \G \lra \G$ be the map defined by $J_{Q}(R) = QRQ^{-1}$. For all $Q \in \G$, the map $\Ad(Q) := d_{e}J_{Q}$ is a Lie algebra automorphism of $\g$. Then, the map $\Ad : \G \lra \mathrm{Aut}(\g)$ is a morphism of groups called the \textit{adjoint representation} of $\G$ and the map $\ad : \g \lra \mathrm{End}(\g)$ defined by $ \ad := d_{e}\Ad$ is a Lie algebra homomorphism called the \textit{adjoint representation} of $\g$. By Theorem 2.8. in \cite{Arvanitoyeorgos 2003},
\begin{equation}\label{adBracket}
(\ad(u))(v) = \lb u, v \rb, \quad u, v \in \g. 
\end{equation}

\subsection{Naturally reductive spaces}

The homogeneous space $\G / \K$ is called \textit{reductive} if there exists a decomposition $\g = \fk \oplus \fm$, where  $\fm$ is an $\Ad(\K)$-invariant subspace called a Lie subspace. Then, the restriction of $d_{e}\psi$ to $\fm$ is a linear isomorphism: 
\begin{equation}\label{tangentSpaceO}
\left( d_{e}\psi \right)_{| \fm}: 
\fm \simeq T_{o}\G / \K. 
\end{equation}

\begin{prop}\label{scaProdMetric}
On a reductive space $\G / \K$, requiring that the isomorphism in $(\ref{tangentSpaceO})$ is an isometry establishes a one-to-one correspondence between $\Ad(\K)$-invariant scalar products on $\fm$ and $\G$-invariant metrics on $\G / \K$.  
\end{prop}

\begin{proof}
See Proposition 22 p.311 in \cite{O'Neill 1983}.
\end{proof}

\begin{defi}

We say that $\G / \K$ is \textit{naturally reductive} if,  $\G / \K$ is reductive and additionally, 
\begin{equation}\label{condNatRed}
\left\langle \lb u, v \rb_{\fm} , w \right\rangle +
\left\langle v, \lb u, w \rb_{\fm} \right\rangle = 0, \quad u, v, w \in \fm.  
\end{equation}

\end{defi}

\noindent
The following Proposition describes the geometry of naturally reductive spaces. 

\begin{prop}
Let $\G / \K$ be a naturally reductive space. Then, its Levi-Civita connection is 
\begin{equation}\label{LCCNatRed}
\left( \nabla_{\widetilde{u}} \widetilde{v} \right)_{o} = -\frac{1}{2} d_{e}\psi \left(\lb u, v \rb_{\fm} \right), \quad u,v \in \fm
\end{equation}

\noindent
and its sectional curvature is 
\begin{equation}\label{CurvNatRed}
K_{o}\left( \widetilde{u} , \widetilde{v} \right) = \frac{1}{4} \left\| \lb u,v \rb_{\fm} \right\|^{2} +
\left\langle \lb \lb u,v \rb_{\fk}, u \rb_{\fm}, v \right\rangle.
\end{equation}
\end{prop}

\begin{proof}
We refer respectively to Proposition 5.2 and Theorem 5.3 in \cite{Arvanitoyeorgos 2003}. 
\end{proof}

\subsection{Normal homogeneous spaces}

In this subsection, we assume that $\G$ is in addition \textit{compact} and that $g^{\G}$ is a \textit{bi-invariant} metric on $\G$. Let $\K$ be a closed subgroup of $\G$. Then, $\G/\K$ is called a \textit{normal homogeneous space} when $\G/\K$ is endowed with the quotient metric of $g^{\G}$ by the right action of $\K$, denoted by $\overline{g}^{\G}$. If we assume additionally that $\G$ is \textit{semisimple}, then one can take for $g^{\G}$ the bi-invariant metric induced by any negative multiple of the Killing form. In this case, we say that $\G / \K$ is a \textit{standard normal} homogeneous space. The following result allows to study normal homogeneous spaces through their Lie algebras. 

\begin{prop}\label{biInvScal}
Any bi-invariant metric on a Lie group $\G$ is induced by a unique $\Ad(\G)$-invariant scalar product on its Lie algebra  $\g$. This last condition is equivalent to 
\begin{equation}\label{AdInvariance} 
\left\langle \lb u,v \rb, w \right\rangle = \left\langle u, \lb v,w \rb \right\rangle, \quad u, v, w \in \g. 
\end{equation}

\end{prop}

\begin{proof}
See Proposition $3.9.$ in \cite{Arvanitoyeorgos 2003}. 
\end{proof}

\begin{corr}\label{corrNormalHom}
$(i)$ A normal homogeneous space $\G/\K$ is naturally reductive, with Lie subspace $\fm_{n}$, which is the orthogonal of $\fk$ in $\g$ wrt $\left\langle \cdot, \cdot \right\rangle$ i.e. 
\begin{equation*}
\fm_{n} = \fk^{\perp} = \left\{ u \in \g : \left\langle u, \fk \right\rangle = 0 \right\}. 
\end{equation*}

\noindent
$(ii)$ The quotient metric $\overline{g}^{\G}$ is a $\G$-invariant metric on $\G / \K$. 

\end{corr}

\begin{proof}
By Proposition $\ref{biInvScal}$, $\left\langle \cdot, \cdot \right\rangle := g^{\G}_{e}$ is $\Ad(\G)$-invariant on $\g$. Now, for all $Q \in \K$, $(\Ad(Q))(\fk)=\fk$. Therefore, $\fm_{n}$ is stable under $\Ad(\K)$, so that $\G / \K$ is \textit{reductive}. Then, $\G/\K$ is indeed naturally reductive, since $(\ref{AdInvariance})$ implies that $(\ref{condNatRed})$ holds. Then, we prove the second assertion. By defintion of the quotient metric $\overline{g}^{\G}$, the map 
\begin{equation*}
\left( d_{e}\psi \right)_{| \fm_{n}}: \left( \fm_{n}, \left\langle \cdot, \cdot \right\rangle_{|\fm_{n}} \right)
 \lra \left( T_{o}\G / \K, \overline{g}^{\G} \right)
\end{equation*}

\noindent
is an isometry. So, by Proposition $\ref{scaProdMetric}$, $\overline{g}^{\G}$ is a $\G$-invariant metric on $\G / \K$.
\end{proof}

\noindent\\
Since such a space is naturally reductive, its Levi-Civita connection is given by $(\ref{LCCNatRed})$, where the Lie subspace is $\fm_{n}$. By definition of $\fm_{n}$ and $(\ref{kerTangentPi})$, we have that $\g = \fm_{n} \oplus \fk = \fm_{n} \oplus \mathrm{ker} (d_{e}\psi)$. So, 
\begin{equation}\label{LCCnormal}
\left( \nabla_{\widetilde{u}} \widetilde{v} \right)_{o} = -\frac{1}{2} d_{e}\psi \left(\lb u, v \rb_{\fm_{n}} \right) = -\frac{1}{2} d_{e}\psi \left(\lb u, v \rb \right). 
\end{equation}

\noindent
The sectional curvature is computed by simplifying $(\ref{CurvNatRed})$, using that $\langle \fm_{n}, \fk \rangle = 0$ and $(\ref{AdInvariance})$. So, 
\begin{equation}\label{curvaNormal}
K_{o}(\widetilde{u} , \widetilde{v}) = \left\| \lb u,v \rb_{\fk} \right\|^{2} + \frac{1}{4}\left\| \lb u,v \rb_{\fm_{n}} \right\|^{2}. 
\end{equation}

\subsection{Symmetric spaces}

\subsubsection{Locally symmetric spaces}

First, we introduce locally symmetric spaces, although they are not in general homogeneous. For $x \in \M$, denote by $\inj_{x}$ the injectivity radius at $x$. The geodesic symmetry through $x$ is the map $s_{x}$ defined by  
\begin{equation}\label{defGeoSym}
s_{x} : B(x, \inj_{x}) \lra  B(x, \inj_{x}), \quad 
s_{x}(y)=\Exp_{x}^{\M} \left( - \left(\Exp_{x}^{\M}\right)^{-1} \left( y \right) \right).
\end{equation}

\noindent
By definition, a locally symmetric space is a Riemannian manifold $(\M, g)$ such that for all $x \in \M$, $s_{x}$ is an isometry on a neighbourhood of $x$.

\subsubsection{Symmetric spaces}

A Riemannian globally symmetric space is a connected Riemannian manifold $(\M, g)$ such that for all $x \in \M$, there exists an isometry $s_{x}$ of $\M$ such that $s_{x}(x) = x$ and $d_{x}s_{x} = - \Id_{T_{x}\M}$. Such a space, simply called a \textit{symmetric space} is a locally symmetric one, and is complete. Let $\G:=\mathrm{Isom}_{0}(\M)$ be the connected component of the identity in the group of isometries of $\M$. Then, $\G$ is a connected Lie group which acts transitively on $\M$. So, $\M$ is identified with a \textit{homogeneous space} $\G / \K$. For fixed $x \in \M$, the map $\sigma_{x} : \G \lra \G$ defined by $Q \lmp s_{x} \circ Q \circ s_{x}$ is an involutive automorphism of $\G$ called a \textit{Cartan involution}, such that $\left(\G_{\sigma_{x}}\right)_{0} \subset \K \subset \G_{\sigma_{x}}$, where $\K$ is the isotropy group of $x$ and $\G_{\sigma_{x}}:=\mathrm{Fix}(\sigma_{x})=\left\{ u \in \g : \sigma_{x}(g)= g \right\}$. In fact, by results of \cite{O'Neill 1983}, a symmetric space is equivalent to a \textit{symmetric data} $(\G / \K, \sigma, B)$ defined below.  

\begin{defi}\label{symData}
A symmetric data is a triple $(\G / \K, \sigma, B)$, where: 

\noindent
$(i)$ $\G$ is a connected Lie group and $\K$ a closed subgroup. 

\noindent
$(ii)$ $\sigma$ is an involutive automorphism of $\G$ such that $\left(\G_{\sigma}\right)_{0} \subset \K \subset \G_{\sigma}$, where $\G_{\sigma}:=\mathrm{Fix}(\sigma)$. 

\noindent
$(iii)$ $B$ is an $\Ad(\K)$-invariant scalar product on $\fm_{s}$, where
\begin{equation*}
\fm_{s}:= \left\{ u \in \g : d_{e}\sigma(u)= - u \right\}.
\end{equation*}
\end{defi}

\noindent
Denoting by $\fk$ the Lie algebra of $\K$, by \cite[Lemma 30]{O'Neill 1983}, $\fk = \left\{ u \in \g : d_{e}\sigma(u)= u \right\}$. Now, $d_{e} \sigma$ is involutive, so that $\g = \fk \oplus \fm_{s}$. Furthermore, $\fm_{s}$ is an $\Ad(\K)$-invariant subspace, which implies that $\G / \K$ is \textit{reductive}. Finally, the following relations hold:
\begin{equation}\label{threeConditions}
\lb \fk, \fk \rb \subset \fk, \quad
\lb \fk, \fm_{s} \rb \subset \fm_{s}
\quad \textrm{and} \quad
\lb \fm_{s}, \fm_{s} \rb \subset \fk. 
\end{equation}

\noindent
The condition $\lb \fm_{s}, \fm_{s} \rb \subset \fk$ implies that $(\ref{condNatRed})$ holds. So, \textit{a symmetric space is naturally reductive}.

\subsubsection{Normal symmetric spaces}\label{symIsNormal}

We considered two classes of naturally reductive spaces : normal homogeneous spaces and symmetric spaces. In this subsection, we investigate their intersection. 

\noindent\\
Let $\G, \K, \sigma$ be as in $(i)$ and $(ii)$ of Definition $\ref{symData}$. Let $\overline{B}$ be an $\Ad(\G)$-invariant scalar product on $\g$ such that its restriction $\overline{B}_{|\fm_{s}}$ to $\fm_{s}$ is an $\Ad(\K)$-invariant scalar product on $\fm_{s}$. If $\G$ is compact, then $(\G / \K, \sigma, \overline{B}_{|\fm_{s}})$ is a symmetric data such that $(\G / \K, \overline{B})$ induces a structure of normal homogeneous space. We say that $\G / \K$ is a \textit{normal symmetric space}. By the Remark p.319 in \cite{O'Neill 1983}, $\overline{B}_{|\fm_{s}}$ is an $\Ad(\K)$-invariant scalar product on $\fm_{s}$ provided that $\overline{B}$ is invariant under $d_{e}\sigma$. In this case, we say that the triple $(\G / \K, \sigma, \overline{B}_{|\fm_{s}})$ is a \textit{usual normal symmetric space}. Thus, $\fm_{s} \perp \fk$ wrt $\overline{B}$ and
\begin{equation}\label{MsPerpK}
\fm_{s} = \fm_{n} := \fk^{\perp}. 
\end{equation}

\begin{rem}\label{KillingBothInva}
If $\G$ is semisimple, then its Killing form is a scalar product on $\g$ which is invariant under both $\Ad(\G)$ and $d_{e}\sigma$. 
\end{rem}

\subsubsection{Geometry of symmetric spaces}

The formulas below are derived form the natural reductivity of symmetric spaces. For the Levi-Civita connection, the condition $\lb \fm_{s}, \fm_{s} \rb \subset \fk$ allows to simplify $(\ref{LCCNatRed})$, so that for $u,v \in \fm_{s}$,
\begin{equation}\label{LCCSym}
\left( \nabla_{\widetilde{u}} \widetilde{v} \right)_{o} = 0. 
\end{equation}

\noindent
For the sectional curvature, we simplify $(\ref{CurvNatRed})$ by using $(\ref{threeConditions})$, to obtain that for $u,v \in \fm_{s}$, 
\begin{equation}\label{curvaSym}
K_{o}\left( \widetilde{u} , \widetilde{v} \right) = 
\left\langle \lb \lb u,v \rb, u \rb, v \right\rangle
\end{equation}

\subsection{Summary}

\subsubsection{Hierarchy of homogeneous spaces}

For $n>2$, $SO(n)$ is semisimple. Indeed, its Killing form is a negative multiple of the Frobenius inner product. So, Grassmannians and flag manifolds are \textit{standard normal} homogeneous spaces. By Remark $\ref{KillingBothInva}$, Grassmannians are usual \textit{normal symmetric} spaces. In contrast, flag manifolds are \textit{not locally symmetric}. More generally, we have seen that normal homogeneous spaces and symmetric spaces are \textit{naturally reducive}. The  diagram of Figure $\ref{figureHierarchy}$ describes the hierarchy of homogeneous spaces aforementioned and locate precisely Grassmannians and flag manifolds within this hierarchy. 

\begin{figure}
\centering
\includegraphics[width=7cm]{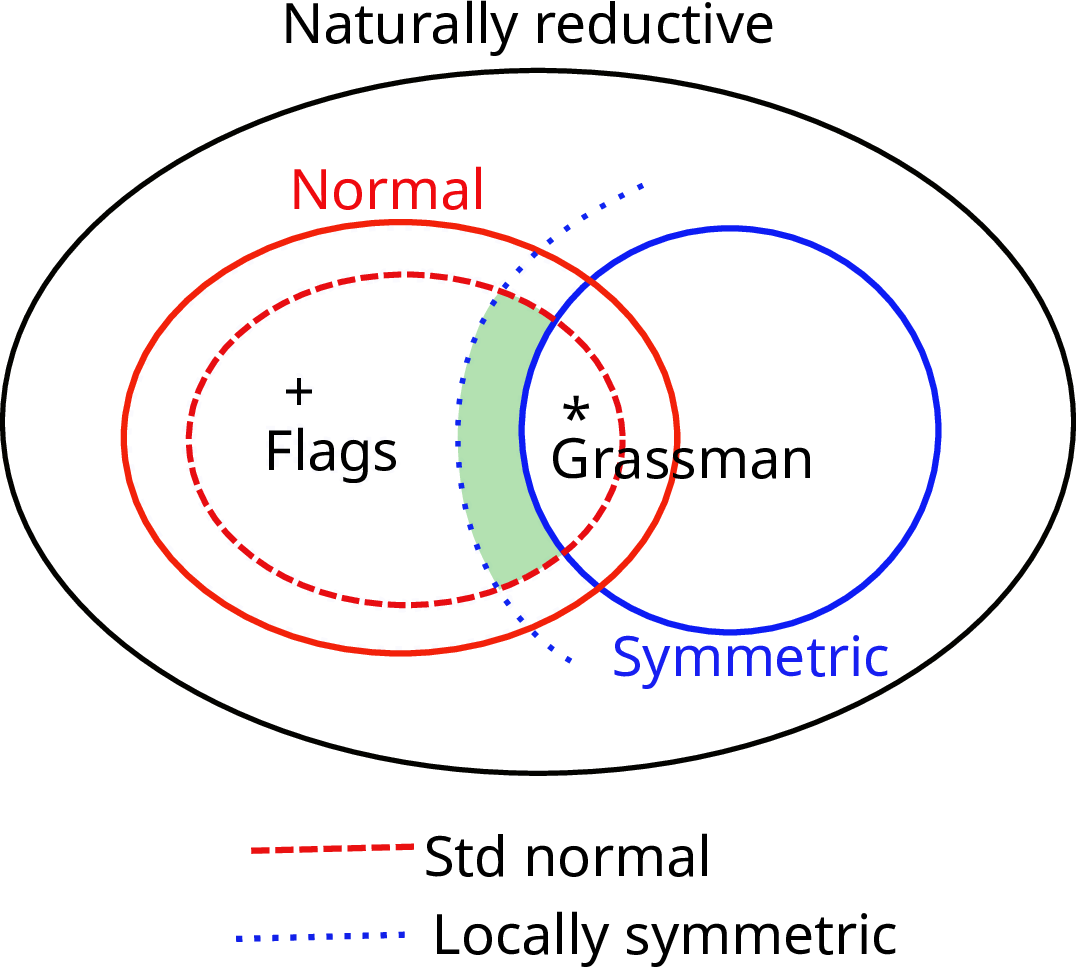}
\caption{Hierarchy of homogeneous spaces\label{figureHierarchy}}
\end{figure}

\noindent\\
In particular, the region in green in this diagram represents the class of standard homogeneous spaces that are locally but not globally symmetric. Thus, the spaces within this class are, in this sense, \textit{intermediate} between Grassmannians and flag manifolds. An example of such a space is provided in \cite{Fels 1997}. Namely, let $S^{3 }= SO(4) / SO(3)$ be the sphere of dimension $3$. Now, $SU(2)$ is a closed subgroup of $SO(4)$ which acts transitively and freely on $S^{3 }$. Then, for any discrete subgroup $\Gamma \subset SU(2)$ with more than $2$ elements, the manifold $\M := SU(2) / \Gamma$ belongs to this class. We refer to \cite{Fels 1997} for a detailed justification.

\subsubsection{Table of formulas}

\noindent\begin{tabular}{|p{2cm}|p{4cm}|p{4.3cm}|p{3.5cm}|}
\hline 
 & \begin{center} Naturally reductive \end{center} & \begin{center} Normal \end{center} & \begin{center} Symmetric \end{center} \\
\hline
Lie subspace $\fm$ such that \newline $\g = \fk \oplus \fm$  &  $\fm$ is an $\Ad(\K)$-invariant subspace of $\g$  &  
\begin{center}
$\fm_{n}:= \fk^{\perp}$
\end{center} &  
\begin{center}
$\fm_{s}:= \left\{ u \in \g : d_{e}\sigma(u)= - u \right\}$
\end{center}  \\
\hline
Relations \newline in $\g$ & For $u, v, w \in \fm$, 
\begin{equation*}
\left\langle \lb u,v \rb_{\fm}, w \right\rangle = \left\langle u, \lb v,w \rb_{\fm} \right\rangle 
\end{equation*} &
For $u, v, w \in \g$
\begin{equation*}
\left\langle \lb u,v \rb, w \right\rangle = \left\langle u, \lb v,w \rb \right\rangle 
\end{equation*}  &  
\begin{center}
$\lb \fm_{s}, \fm_{s} \rb \subset \fk$
\end{center} \\
\hline
Levi-Civita connection at origin $o$ & 
For $u, v \in \fm$, \begin{equation*}
\left( \nabla_{\widetilde{u}} \widetilde{v} \right)_{o} = -\frac{1}{2} d_{e}\psi \left(\lb u, v \rb_{\fm} \right)
\end{equation*}  & 
For $u, v \in \fm_{n}$, \begin{equation*}
\left( \nabla_{\widetilde{u}} \widetilde{v} \right)_{o} = -\frac{1}{2} d_{e}\psi \left(\lb u, v \rb \right)
\end{equation*} & 
For $u, v \in \fm_{s}$, 
\begin{equation*}
\left( \nabla_{\widetilde{u}} \widetilde{v} \right)_{o} = 0
\end{equation*}\\
\hline
Sectional curvature \newline at origin $o$ & 
For $u,v \in \fm$, 
\begin{multline*}
K_{o}\left( \widetilde{u} , \widetilde{v} \right) = \frac{1}{4} \left\| \lb u,v \rb_{\fm} \right\|^{2}  \\
+ \left\langle \lb \lb u,v \rb_{\fk}, u \rb_{\fm}, v \right\rangle
\end{multline*}  & 
For $u, v \in \fm_{n}$, 
\begin{multline*} 
K_{o}\left( \widetilde{u} , \widetilde{v} \right) = \frac{1}{4}\left\| \lb u,v \rb_{\fm_{n}} \right\|^{2}\\ 
+ \left\| \lb u,v \rb_{\fk} \right\|^{2} 
\end{multline*}  & 
For $u,v \in \fm_{s}$, 
\begin{multline*}
K_{o}\left( \widetilde{u} , \widetilde{v} \right) = \\
\left\langle \lb \lb u,v \rb, u \rb, v \right\rangle
\end{multline*}  \\
\hline
\end{tabular}

\subsubsection{Comparison of formulas}\label{compareFormulas1}

\begin{defi}
 We say that a normal homogeneous space $\G/\K$ satisfies the Bracket Condition when   
\begin{equation}\label{BracketConditionNHS}
\lb \fm_{n}, \fm_{n} \rb \subset \fk.
\end{equation}
\end{defi}

\noindent
We introduce this condition because, when it holds, it allows to simplify the formulas of the connection and the curvature of normal homogeneous spaces. Actually, these simplified formulas coincide with those for usual normal symmetric spaces, as synthetized in the following table.

\bigskip

\noindent\begin{tabular}{|p{2cm}|p{5cm}|p{7cm}|}
\hline 
 &  \begin{center} Normal + $[\fm_{n}, \fm_{n}]\subset \fk$ \end{center} & \begin{center} Usual normal symmetric  \end{center} \\
\hline
Levi-Civita connection at $o$ & 
For $u, v \in \fm_{n}$, \begin{equation*}
\left( \nabla_{\widetilde{u}} \widetilde{v} \right)_{o} = 0
\end{equation*} & 
For $u, v \in \fm_{s}=\fk^{\perp}$, \begin{equation*}
\left( \nabla_{\widetilde{u}} \widetilde{v} \right)_{o} = 0 
\end{equation*} \\
\hline
Sectional curvature \newline at $o$ & 
For $u, v \in \fm_{n}$, 
\begin{equation*} 
K_{o}\left( \widetilde{u} , \widetilde{v} \right) = \left\| \lb u,v \rb_{\fk} \right\|^{2} 
= \left\| \lb u,v \rb \right\|^{2}
\end{equation*}  & 
For $u,v \in \fm_{s}=\fk^{\perp}$, 
\begin{equation*}
K_{o}\left( \widetilde{u} , \widetilde{v} \right) = \left\langle \lb \lb u,v \rb, u \rb, v \right\rangle 
= \left\| \lb u,v \rb \right\|^{2}
\end{equation*}  
Follows from the relation in $\g$ for the normal case.\\ 
\hline
\end{tabular}

\bigskip

\begin{rem}
We notice that usual normal symmetric spaces satisfy the Bracket Condition $(\ref{BracketConditionNHS})$, which follows by combining the relations $\fm_{n}=\fm_{s}$ and $\lb \fm_{s}, \fm_{s} \rb \subset \fk$.  Conversely, this condition is not sufficient to characterize normal symmetric spaces among normal homogeneous ones. However, if one assumes additionally that $\fk$ is a compactly embedded Lie subalgebra of $\g$, then $\G / \K$, endowed with any $\G$-invariant metric, is a symmetric space. We refer to \cite{Helgason 1978} for the concept of compactly embedded Lie subalgebra. 
\end{rem}

\begin{rem}
We check in Section $\ref{sect4}$ that Grassmannians satisfy indeed the Bracket Condition $(\ref{BracketConditionNHS})$, contrarily to flag manifolds. 
\end{rem}

\section{Effective formulas for realizations of normal homogeneous spaces}\label{sect3}

We recall and precise the notations of the Introduction. $\M$ is a smooth manifold, on which a compact connected Lie group $\G$, with bi-invariant metric $g^{\G}$, acts smoothly. $\B$ is the orbit of $b_{0} \in \M$, with isotropy group $\K$. The main orbital map $\pi_{0} : \G \lra \B$ is defined by $\pi_{0}(Q)=Q \cdot b_{0}$. 

Let $\psi : \G \lra \G/\K$ be the canonical quotient map and $\overline{g}^{\G}$ the quotient metric on $\G / \K$. So, $\G / \K$ is a \textit{normal homogeneous space}. $\B$ is endowed with the metric $g^{\B}$ for which the map $\widehat{\pi_{0}} : \left( \G/\K, \overline{g}^{\G} \right) \lra \left(\B, g^{\B} \right)$, where $\pi_{0}=\widehat{\pi_{0}} \circ \psi$, is an isometry. Let $\fk$ be the Lie algebra of $\K$. By $(\ref{kerTangentPi})$, $\fk=\ker(d_{e}\psi)=\ker(d_{e}\pi_{0})$. We denote by $\fm$ the Lie subspace of $\G / \K$. Then, 
\begin{equation*}
\fm := \fk^{\perp} = \left(\ker(d_{e}\pi_{0}) \right)^{\perp}. 
\end{equation*}

\noindent
By construction, the main orbital map is a \textit{Riemannian submersion}. We establish explicit formulas for the connection and the curvature of $\B$. For the connection, we prove in subsection $\ref{PrelimDecVF}$, the existence of a neighborhood of $b_{0}$ on which any vector field on $\B$ is decomposed into a linear combination of fundamental ones (see $(\ref{defFVFonB})$ below), with \textit{smooth} coefficients. Such a decomposition allows to derive, in subsection $\ref{subsecLCCofB}$, the Levi-Civita connection of $\B$ for arbitrary vector fields from that for $\G/\K$ recalled in the preceding Section, which hold only for fundamental vector fields on $\G/\K$. We recall the Bracket Condition for $\G / \K$, under which these formulas are simplified:
\begin{equation}\label{BracketCondB}
\lb \fm, \fm \rb \subset \fk. 
\end{equation}

\subsection{Structure of the main orbital map}

In this subsection, we essentially develop some consequences of the Riemannian submersion structure of $\pi_{0}$. However, we state hereafter another result on its rich structure, which is exploited in subsection $\ref{subsecPertFlags}$ below.  

\begin{lem}\label{structureMainOM}
$\left(\G, \pi_{0}, \B, \K \right)$ is a principal bundle. 
\end{lem}

\begin{proof}
$\pi_{0}$ is a surjective submersion and the Lie group $\K$ acts freely on $\G$, by right-multiplication, such that the orbits of this action are exactly the fibers of $\pi_{0}$. So, by Lemma $18.3$ in \cite{Michor 2008}, $\left(\G, \pi_{0}, \B, \K \right)$ is a principal bundle. 
\end{proof}

\subsubsection{Metric and geodesics of $\B$}

First, the \textit{Riemannian submersion} structure of $\pi_{0}$ provides readily the description of the metric and the geodesics of $\B$. For $Q \in \G$, let $\cL_{Q}$ (resp. $\cR_{Q}$) be the left (resp. right) multiplication by $Q$ in $\G$, which are isometries of $\left(\G , g^{\G} \right)$. For any $R \in \G$ and $\Delta \in T_{R}\G$, set 
\begin{equation*}
\cL_{Q}[\Delta] := \left( d_{R}\cL_{Q} \right)(\Delta) \in T_{QR}\G 
\quad \textrm{and} \quad 
\cR_{Q}[\Delta] := \left( d_{R}\cR_{Q} \right)(\Delta) \in T_{RQ}\G. 
\end{equation*}

\noindent
For any $Q \in \G$, $T_{Q}\G$ is endowed with the inner product $g^{\G}_{Q}$. Then, set
\begin{equation}\label{HorVerRS}
V_{Q}:=\ker(d_{Q}\pi_{0}) \quad \textrm{and} \quad H_{Q}:=V_{Q}^{\perp}.
\end{equation}

\noindent
Then, $\left( d_{Q}\pi_{0} \right)_{|H_{Q}}$ is a linear isometry. For $b \in \B$, $\Delta \in T_{b}\B$ and $Q \in \pi_{0}^{-1}(b)$, the \textit{horizontal lift} at $Q$ of $\Delta$ wrt $\pi_{0}$ is 
\begin{equation}\label{horLiftClassic}
\Delta^{\sharp}_{Q} := \left( d_{Q}\pi_{0} \right)_{|H_{Q}}^{-1}\left( \Delta \right). 
\end{equation}

\noindent
So, the metric $g^{\B}$ is defined as follows: 
\begin{equation}\label{metricB}
g^{\B}_{b}\left( \Delta, \Delta' \right) = g^{\G}_{Q}\left( \Delta^{\sharp}_{Q}, (\Delta')^{\sharp}_{Q} \right), \quad \Delta, \Delta' \in T_{b}\B. 
\end{equation}

\noindent
For any manifold $\M$, we denote by $\gamma^{\M}_{x,v}$ the geodesic in $\M$ through $x$ in the direction $v$. Since $\pi_{0}$ is a Riemannian submersion, the geodesics in $\B$ are the images by $\pi_{0}$ of \textit{horizontal geodesics} in $\G$, i.e. 
\begin{equation*}
\gamma^{\B}_{b,\Delta} = \pi_{0}\left( \gamma^{\G}_{Q,\Delta^{\sharp}_{Q}} \right), \quad b \in \B, \Delta \in T_{b}\B \textrm{ and } Q \in \pi_{0}^{-1}(b). 
\end{equation*}
 
\noindent
Now, $g^{\G}$ is bi-invariant. So, the geodesics in $\G$ are left translates of its one parameter subgroups:  
\begin{equation}\label{geodesicB}
\gamma^{\B}_{b,\Delta} : t \lmp \pi_{0} \left( Q \exp_{\g} \left( t \cL_{Q^{-1}} \left[\Delta^{\sharp}_{Q}\right] \right) \right), 
\end{equation}

\noindent
where $\exp_{\g} : \g \longrightarrow \G$ is the Lie group exponential map of $\G$.

\subsubsection{Sections of Riemannian submersions}

Let $\Pi : \left( \E, g^{\E} \right) \lra \left( \B, g^{\B} \right)$ be a Riemannian submersion, with $\E$ complete and $\B$ connected. Set  
\begin{equation*}
\cV_{\beta}:=\B \setminus \Ct(\beta), \quad \beta \in \B. 
\end{equation*}

\noindent
where $\Ct(\beta)$ is the cut locus of $\beta$. Since $\B$ is also complete, for all $b \in \cV_{\beta}$, there exists a unique minimal geodesic $\gamma(\beta, b)$ between $\beta$ and $b$. For $x \in \Pi^{-1}(\beta)$, let $\gamma(\beta, b)^{\sharp}_{x}$ be the horizontal lift through $x$ of $\gamma(\beta,b)$. Then, $\gamma(\beta, b)^{\sharp}_{x}$ is a horizontal geodesic in $\E$ whose endpoint, denoted by $\cH^{\Pi}_{(\beta, b)}(x)$, lies in the fiber of $b$. Namely, 
\begin{equation}\label{holonomyFormula}
\cH^{\Pi}_{(\beta, b)}(x) = \Exp^{\E}_{x} \left( \left( \Log^{\B}_{\beta}(b) \right)^{\sharp}_{x} \right), \quad x \in \Pi^{-1}(\beta).
\end{equation}

\noindent
Then, the map 
\begin{equation}\label{defHolMap}
\cH^{\Pi}_{(\beta, b)} : \Pi^{-1}(\beta) \lra \Pi^{-1}(b), \qquad x ~\lmp~ \cH^{\Pi}_{(\beta, b)}(x)
\end{equation}

\noindent
is a diffeomorphism, called a \textit{holonomy map} (see \cite{Hermann 1960} or \cite{Guijarro and Walschap 2007}). Thus, we proved the following result.

\begin{lem}\label{lemLocSection}
Let $\Pi : \left( \E, g^{\E} \right) \lra \left( \B, g^{\B} \right)$ be a Riemannian submersion, with $\E$ complete and $\B$ connected. For fixed $\beta \in \B$, set $\cV_{\beta}:=\B \setminus \Ct(\beta)$ and for $x_{\beta} \in \Pi^{-1}\left( \beta \right)$, consider the map 
\begin{equation*}
\mathfrak{S}_{\beta}^{x_{\beta}} : \cV_{\beta} \lra \E, 
\quad
\mathfrak{S}_{\beta}^{x_{\beta}}\left( b \right) = \cH^{\Pi}_{(\beta, b)}(x_{\beta}).
\end{equation*}

\noindent
Then, $\mathfrak{S}_{\beta}^{x_{\beta}}$ is a smooth local section of $\Pi$. 
\end{lem}

\smallskip

\begin{corr}\label{corLocSection}
The map $\sigma_{0}$ defined hereafter is a smooth local section of $\pi_{0}$ such that $\sigma_{0}(b_{0})=e$. 
\begin{equation}\label{defSigma0}
\sigma_{0} : \cV_{b_{0}} \lra \G, \quad \sigma_{0}(b)=\cH^{\pi_{0}}_{(b_{0}, b)}(e).
\end{equation}

\end{corr}

\smallskip

\begin{rem}
Since $\B$ is compact, we can consider a finite open cover $\left( \cV_{\beta_{j}} \right)_{j \in J}$ of $\B$, with $\left| J \right|<\infty$. For any $j \in J$, the map $\phi_{j} : \cV_{\beta_{j}} \times \K \lra \pi_{0}^{-1}\left( \cV_{\beta_{j}} \right)$ defined by $\phi_{j}\left( b, k \right) = \left( \mathfrak{S}_{\beta_{j}}^{x_{\beta_{j}}} \left( b \right) \right) k$ is a diffeomorphism. Then, the family $\left( \cV_{\beta_{j}}, \phi_{j}^{-1} \right)_{j \in J}$ is a principal bundle atlas for $\left(\G, \pi_{0}, \B, \K \right):$ See the proof of Lemma $18.3$ in \cite{Michor 2008}.
\end{rem}

\subsection{Preliminary results}\label{PrelimDecVF}

Set $N:=\dim(\B)= \dim(\fm)$ and $K:=\dim(\fk)$. Thus, $\dim(\G)=N+K$. In the sequel, we fix an orthonormal basis $(\epsilon_{k})_{1 \leq k \leq N}$ of $\left( \fm , \langle \cdot , \cdot \rangle \right)$.

\subsubsection{Translations in $\B$}

We call translations in $\B$ the maps $\left(\cL^{\B}_{Q}\right)_{Q \in \G}$ defined by 
\begin{equation*}
\cL^{\B}_{Q} :\B \lra \B, \quad \left(\cL^{\B}_{Q}\right)(b) = Q \cdot b. 
\end{equation*}

\begin{lem}
For any $Q \in \G$, $\cL^{\B}_{Q}$ is an isometry of $\left( \B, g^{\B} \right)$. 
\end{lem}

\begin{proof}
By $(ii)$ of Corollary $\ref{corrNormalHom}$, $\overline{g}^{\G}$ is a $\G$-invariant metric on $\G/\K$. So, the result holds because 
\begin{equation*}
\widehat{\pi_{0}} \circ \cL^{\G/\K}_{Q} = \cL^{\B}_{Q} \circ \widehat{\pi_{0}},
\end{equation*}

\noindent
which follows from the definition of a group action. 
\end{proof}

\noindent
Throughout the sequel, for $b \in \B$ and $\Delta \in T_{b}\B$, we set
\begin{equation*}
Q \cdot [\Delta] := \left(d_{b}\cL^{\B}_{Q}\right)(\Delta) \in T_{Q \cdot b}\B.
\end{equation*}

\noindent
Furthermore, we denote by $Q_{b}$ an element of $\G$ such that $b = \cL^{\B}_{Q_{b}}(b_{0})$ i.e. $b=Q_{b} \cdot b_{0}=\pi_{0}(Q_{b})$. Now, the map $\sigma_{0} : \cV_{b_{0}} \lra \G$ of Corollary \ref{corLocSection} is a \textit{smooth section} of $\pi_{0}$. In the sequel, we assume that 
\begin{equation}\label{QbEqualSig}
Q_{b}=\sigma_{0}(b), \quad b \in \cV_{b_{0}}.
\end{equation}

\subsubsection{Orbital maps onto $\B$}

For $Q\K \in \G/\K$, the orbital map $\psi_{Q\K} : \G \lra \G/\K$ defined in $(\ref{orbitalMapGK})$ of Section $\ref{sect2}$ is called an orbital map onto $\G/\K$. Similarly, we define orbital maps onto $\B$. Namely, for $b \in \B$, let $\pi_{b} : \G \lra \B$ be the map defined by  
\begin{equation}\label{orbitalMapB}
\pi_{b}(Q) = \cL^{\B}_{Q}(b) = Q \cdot b, \quad Q \in \G. 
\end{equation}

\noindent
Thus, $\pi_{b_{0}}=\pi_{0}$. The orbital maps onto $\G/\K$ and onto $\B$ are linked by $(\ref{linkOrbitalMaps})$ hereafter: 
\begin{equation}\label{linkOrbitalMaps}
\widehat{\pi_{0}} \circ \psi_{Q\K} = \pi_{b}, \quad Q\K \in \G / \K, \enskip b=\widehat{\pi_{0}}(Q\K), 
\end{equation}

\noindent
 which follows from the definitions. Now, we link the orbital maps $\pi_{b}$ to the main one $\pi_{0}$. On the one hand, 
\begin{equation}\label{PIBandR}
\pi_{0} \circ \cR_{Q_{b}} : R \lmp RQ_{b} \lmp R \cdot b, 
\quad \textrm{i.e.} \quad
\pi_{b} = \pi_{0} \circ \cR_{Q_{b}}.
\end{equation} 
 
\noindent
On the other hand, recalling that the map $J_{Q} : \G \lra \G$ is defined by $J_{Q} : R \lmp QRQ^{-1}$, we have:
\begin{equation}\label{PIBandJ}
\cL^{\B}_{Q_{b}} \circ \pi_{0} \circ J_{Q_{b}^{-1}} : R \lmp Q_{b}^{-1}RQ_{b} \lmp \left( Q_{b}^{-1}R \right) \cdot b \lmp R \cdot b, 
\quad \textrm{i.e.} \quad
\pi_{b} = \cL^{\B}_{Q_{b}} \circ \pi_{0} \circ J_{Q_{b}^{-1}}.
\end{equation}
 
\noindent
We differentiate $(\ref{PIBandR})$ and $(\ref{PIBandJ})$ at $e$, and we obtain the following key Lemma. 
 
\begin{lem}\label{MBandMH}
For all $b \in \B$, let $Q_{b} \in \G$ such that $b=Q_{b} \cdot b_{0}$. Then, 
 
\smallskip
 
\noindent
$(i)$ $\fk_{b} = \cR_{Q_{b}^{-1}}\left[V_{Q_{b}}\right]$ and $\fm_{b} = \cR_{Q_{b}^{-1}}\left[H_{Q_{b}}\right]$.

\smallskip

\noindent
$(ii)$ $\fk_{b}=\Ad(Q_{b})(\fk)$ and $\fm_{b}=\Ad(Q_{b})(\fm)$.
\end{lem}

\begin{proof}
$(i)$ By $(\ref{PIBandR})$, $d_{e}\pi_{b} = d_{Q_{b}}\pi_{0} \circ d_{e}\cR_{Q_{b}}$. So, $\fk_{b} = \left(d_{e}\cR_{Q_{b}} \right)^{-1} (V_{Q_{b}}) = \cR_{Q_{b}^{-1}}\left[V_{Q_{b}}\right]$. Since $\cR_{Q_{b}^{-1}}$ is an isometry of $\left( \G, g^{\G} \right)$, we deduce that $\fm_{b} = \cR_{Q_{b}^{-1}}\left[H_{Q_{b}}\right]$. 

\smallskip

\noindent
$(ii)$ By $(\ref{PIBandJ})$, $d_{e}\pi_{b}= d_{b_{0}}\cL^{\B}_{Q_{b}} \circ d_{e}\pi_{0} \circ \Ad\left(Q_{b}^{-1}\right)$. Therefore, $\fk_{b}=\Ad(Q_{b})(\fk)$. By Proposition $\ref{biInvScal}$, $\left\langle \cdot, \cdot \right\rangle$ is $\Ad(\G)$-invariant on $\g$, which implies that $\fm_{b}=\Ad(Q_{b})(\fm)$. 
\end{proof}
 
\subsubsection{Fundamental vector fields on $\B$}

Recall that for $u \in \g$, the \textit{fundamental vector field} $\widetilde{u} \in \X(\G / \K)$ is defined by 
$\widetilde{u}(Q\K) = \left(d_{e} \psi_{Q\K}\right)(u)$, for $Q\K \in \G / \K$. Similarly, we define a fundamental vector field $u^{*} \in \X(\B)$ by 
 \begin{equation}\label{defFVFonB}
u^{*}(b) = \left(d_{e} \pi_{b}\right)(u) \in T_{b}\B, \quad b \in \B. 
\end{equation}

\noindent
Then, we precise the link between these fundamental vector fields on $\G / \K$ and on $\B$. Namely, 
\begin{equation}\label{linkFVFs}
u^{*}=\left( \widehat{\pi} \right)_{*} \widetilde{u}, \quad u \in \g, 
\end{equation}

\noindent
which is a direct consequence of the definition of orbital maps and of $(\ref{linkOrbitalMaps})$.

\subsubsection{Lifts to the Lie subspace}

Here, we introduce the maps which relate vector fields on $\B$ to the Lie subspace $\fm$ in $\g$. Set   
\begin{equation*}
\fk_{b}:=\ker(d_{e}\pi_{b}) \quad \textrm{and} \quad \fm_{b}:=\fk_{b}^{\perp}, \quad b \in \B. 
\end{equation*}

\begin{lem}\label{UisOpen}
Let $\U$ be the set of elements $b \in \B$ such that $\fm$ is a complement of $\fk_{b}$ in $\g$, i.e. 
\begin{equation*}
\U := \left\{ b \in \B : \fk_{b} \cap \fm = \left\{0\right\} \right\}. 
\end{equation*}

\noindent
Then, $\U$ is an open neighborhood of $b_{0}$ in $\B$. 
\end{lem}

\begin{proof}
See Appendix. 
\end{proof}

\noindent
For all $b \in \B$, the map $d_{e}\pi_{b} : \g \lra T_{b}\B$ is surjective. Then, by Lemma $\ref{surjectiveLinearMap}$ below, for all $b \in \U$, the restriction $\left(d_{e}\pi_{b}\right)_{| \fm} : \fm \lra T_{b}\B$ is an isomorphism.

\begin{lem}\label{surjectiveLinearMap}
Let $E$ and $F$ be vector spaces of finite dimension and $f : E \lra F$ a surjective linear map. Let $E'$ be a linear  subspace of $E$ such that $\dim(E')=\dim(F)$ and $E'$ is a complement of $\ker(f)$ in $E$. Then, the restriction $f_{|E'} : E' \lra F$ is an isomorphism.  
\end{lem}

\noindent
Then, we may introduce the main definitions of this Section.

\begin{defi}
For $b \in \U$ and $\Delta \in T_{b}\B$, the $\fm$-lift of $\Delta$ wrt $\pi_{b}$ is the element $\omega_{b}(\Delta) \in \fm$ defined by 
\begin{equation*}
\omega_{b}(\Delta) := \left(\left( d_{e}\pi_{b} \right)_{|\fm}\right)^{-1} \left( \Delta \right). 
\end{equation*}

\noindent
In particular, with the notations $(\ref{HorVerRS})$, $\fm=H_{e}$. Therefore, for any $\Delta \in T_{b_{0}}\B$,  
\begin{equation}\label{linkFVFandRS}
\omega_{b_{0}}(\Delta) = \left(  \left(d_{e}\pi_{0} \right)_{|\fm}\right)^{-1} \left( \Delta \right) = \Delta^{\sharp}_{e}. 
\end{equation}

\noindent
For $b \in \U$, the map $\omega_{b} : T_{b}\B \lra \fm$ is linear. Thus, the map $\omega : b \in \U \lmp \omega_{b}$ is called the fundamental $1$-form. 
\end{defi}

\begin{defi}\label{defMliftFunc}
For $X \in \X(\B)$, the $\fm$-lift function of $X$ is the map  
\begin{equation*}
\overline{X} : \U \lra \fm,
\quad \overline{X}(b) = \omega_{b}(X(b)). 
\end{equation*}
\end{defi}

\subsubsection{Expression of the $\fm$-lift function}

The smoothness of the $\fm$-lift function is not clear from its definition. We establish in Proposition $\ref{propMliftFunction}$ below an expression of the $\fm$-lift function that allows us to prove its smoothness in a neighborhood of $b_{0}$ in Proposition $\ref{smoothXbar}$ below. First, we introduce preliminary definitions hereafter. 

\begin{defi}
For all $b \in \U$ and $\Delta \in T_{b}\B$, the $\fm_{b}$-lift of $\Delta$ is 
\begin{equation}\label{mbLift}
\Delta^{[b]}:=\left(\left(d_{e}\pi_{b}\right)_{|\fm_{b}}\right)^{-1} (\Delta) \in \fm_{b}.
\end{equation}
\end{defi}

\begin{defi}
For $b \in \B$, let $p_{b} : \fm \lra \fm_{b}$ be the restriction to $\fm$ of the orthogonal projection onto $\fm_{b}$ in $\g$. 
\end{defi}

\smallskip

\noindent
Lemma $\ref{surjectiveLinearMap}$ implies that, for any $b \in \U$, $p_{b}$ is an \textit{isomorphism}. The expression of the $\fm$-lift function 
$\overline{X}$ obtained hereafter involves the inverse of $p_{b}$, which is described in Lemma $\ref{inverseProjection}$ below.

\begin{prop}\label{propMliftFunction}
For $b \in \U$, the $\fm_{b}$-lift of $\Delta  \in T_{b}\B$ is given by 
\begin{equation}\label{DeltaHofB}
\Delta^{[b]} = \cR_{Q_{b}^{-1}}\left[ \Delta^{\sharp}_{Q_{b}} \right].
\end{equation}

\noindent
Then, the $\fm$-lift function of $X \in \X(\B)$ is expressed as 
\begin{equation}\label{mLiftExpression}
\overline{X}(b) = p_{b}^{-1}\left( X(b)^{[b]} \right) = 
p_{b}^{-1}\left( \cR_{Q_{b}^{-1}}\left[ X(b)^{\sharp}_{Q_{b}} \right] \right), \quad b \in \V. 
\end{equation}

\end{prop}

\begin{proof}
By $(i)$ of Lemma $\ref{MBandMH}$, the following diagram is commutative. 

\[\xymatrix{
  \fm_{b} \ar[rr]^{d_{e}\cR_{Q_{b}}} \ar[rrd]_{\left(d_{e}\pi_{b}\right)_{|\fm_{b}}} && H_{Q_{b}} \ar[d]^{d_{Q_{b}}\pi_{0}} \\
  && T_{b}\B  
  }
\]
 
\noindent
So, for any $b \in \U$, 
\begin{equation*}
\left( \left(d_{e}\pi_{b}\right)_{|\fm_{b}}\right)^{-1} = d_{Q_{b}}\cR_{Q_{b}^{-1}} \circ \left( \left(d_{e}\pi_{0}\right)_{|H_{Q_{b}}}\right)^{-1}.
\end{equation*}

\noindent
By evaluating this identity at $\Delta$, we obtain that $(\ref{DeltaHofB})$ holds. Now, by definition, 
\begin{equation*}
\left(d_{e}\pi_{b} \right)\left( \Delta^{[b]} \right) = \left(d_{e}\pi_{b} \right)\left( \omega_{b}(\Delta) \right) = \Delta.
\end{equation*}

\noindent
Thus,  
\begin{equation*}
\left( \Delta^{[b]} - \omega_{b}(\Delta) \right) \in \fk_{b}=\fm_{b}^{\perp}
\quad \textrm{and} \quad
\omega_{b}(\Delta) \in \fm.
\end{equation*}

\noindent
In other words,  
\begin{equation*}
\omega_{b}(\Delta) = p_{b}^{-1}\left( \Delta^{[b]} \right).
\end{equation*}

\noindent
Since $\overline{X}(b) = \omega_{b}(X(b))$, we derive the expression of $\overline{X}(b)$ in $(\ref{mLiftExpression})$.
\end{proof}

\noindent\\
For $b \in \U$, the map $p_{b}^{-1} : \fm_{b} \lra \fm$ is determined as follows. Recall that $(\epsilon_{k})_{1 \leq k \leq N}$ is an orthonormal basis of $\left( \fm , \langle \cdot , \cdot \rangle \right)$. Then, by $(ii)$ of Lemma $\ref{MBandMH}$, an orthonormal basis of $\fm_{b}$ is 
\begin{equation}\label{basisMB}
\left\{ \epsilon'_{k}(b) := \Ad(Q_{b})(\epsilon_{k}) : 1 \leq k \leq N \right\}. 
\end{equation}

\noindent
Now, for all $u \in \fm$, $p_{b}(u) = \sum\limits_{k=1}^{N} \left\langle u , \epsilon'_{k}(b) \right\rangle \epsilon'_{k}(b)$. So, the matrix of the linear map $p_{b} : \fm \lra \fm_{b}$ wrt the basis $(\epsilon_{k})_{k}$ of $\fm$ and $(\epsilon'_{k}(b))_{k}$ of $\fm_{b}$ is the $N \times N$ matrix $\cA(b)$ defined by 
\begin{equation}\label{defMatrixA}
\cA(b) = (\alpha_{k,\ell}(b))_{k, \ell} \quad \textrm{where} \quad
\alpha_{k,\ell}(b) = \left\langle \epsilon_{\ell} , \epsilon'_{k}(b) \right\rangle, \quad 1 \leq k \leq \ell \leq N. 
\end{equation}

\noindent
Then, the matrix $\cA(b)$ is invertible, and we deduce hereafter the expression of $p_{b}^{-1}$. 

\begin{lem}\label{inverseProjection}
Let $b \in \U$. For any $v \in \fm_{b}$, set $u:=p_{b}^{-1}(v) \in \fm$. Then, 
\begin{equation}\label{systemNN}
u = \sum\limits_{k=1}^{N} u_{k} \epsilon_{k}
\quad \textrm{where} \quad
u_{k} = \sum\limits_{\ell=1}^{N} \left( \cA(b)^{-1} \right)_{k, \ell} v_{\ell} = \sum\limits_{\ell=1}^{N} \left( \cA(b)^{-1} \right)_{k, \ell} \left\langle v, \epsilon'_{\ell}(b) \right\rangle. 
\end{equation}
\end{lem}

\subsubsection{Smoothness of the $\fm$-lift function}

In $(\ref{mLiftExpression})$, for $X \in \X(\B)$, we have expressed $\overline{X}$ in function of \textit{horizontal lifts} wrt $\pi_{0}$ and of $Q_{b}$. On the one hand, the Riemannian submersion structure of $\pi_{0}$ implies that the horizontal lifts wrt $\pi_{0}$ of smooth vector fields on $\B$ are \textit{smooth} vector fields on $\G$. On the other hand, $\overline{X}$ is defined on $\U$ and we have assumed in $(\ref{QbEqualSig})$ that the map $b \lmp Q_{b}$ coincides with $\sigma_{0}$ on $\cV_{b_{0}}$, where this map is \textit{smooth}. Thus, these two arguments are used hereafter to prove that $\overline{X}$ is smooth on the open neighborhood $\V$ of $b_{0}$ defined by 
\begin{equation*}
\V := \U \cap \cV_{b_{0}}. 
\end{equation*}

\begin{prop}\label{smoothXbar}
For any $X \in \X(\B)$, its $\fm$-lift function $\overline{X}$ is smooth on $\V$. 
\end{prop}

\begin{proof}
For any manifold $\M$, we denote by $T\M$ its tangent bundle. Let $\cW_{0}$ be the subset of $T\G$ defined by 
\begin{equation*}
\cW_{0} := \left\{ (Q,v) \in T\G : Q \in \sigma_{0}(\V), v \in H_{Q} \right\}. 
\end{equation*}

\noindent
Then, the map $\overline{X} : \V \lra \fm$ is decomposed as follows. 
\begin{equation}\label{decXbarFunctions}
\overline{X} : b \lmp \sigma_{0}(b)=Q_{b} \lmp \left( Q_{b} , X(b)_{Q_{b}}^{\sharp} \right) \lmp \left( b , p_{b}^{-1}\left( \cR_{Q_{b}^{-1}}\left[ X(b)^{\sharp}_{Q_{b}} \right] \right) \right) \lmp \overline{X}(b). 
\end{equation}

\noindent
By $(\ref{mLiftExpression})$, the last map involved in $(\ref{decXbarFunctions})$ is the projection $\mathrm{pr}_{2}$ onto the second factor. Thus,  
\begin{equation*}
\overline{X} = \mathrm{pr}_{2} \circ \Psi \circ \Phi_{X^{\sharp}} \circ \sigma_{0} : \V \lra \sigma_{0}(\V) \lra \cW_{0} \lra \V \times \fm \lra \fm
\end{equation*}

\noindent
where the maps $\Phi_{X^{\sharp}}$ and $\Psi$ are defined respectively by 
\begin{equation*}
\Phi_{X^{\sharp}} : \G \lra T\G, \quad \Phi_{X^{\sharp}}\left(Q \right) = \left( Q , X^{\sharp}(Q) \right)
:= \left( Q , X(\pi_{0}(Q))^{\sharp}_{Q} \right)
\end{equation*}

\noindent
and
\begin{equation*}
\Psi : \cW_{0} \lra \V \times \fm, \quad \Psi(Q,v) = \left( \pi_{0}(Q) , p_{\pi_{0}(Q)}^{-1} \left( \cR_{Q^{-1}}(v) \right) \right).
\end{equation*}

\noindent
Firstly, the horizontal lift wrt $\pi_{0}$ of $X$ is the vector field $X^{\sharp}$ on $\G$ whose value at any $Q \in \G$ is $X^{\sharp}(Q):=X(\pi_{0}(Q))^{\sharp}_{Q}$, involved in the definition of the map $\Phi_{X^{\sharp}}$. Now, it is well-known that $X^{\sharp}$ is a smooth vector field on $\G$, which means that the map $\Phi_{X^{\sharp}}$ is \textit{smooth}. Secondly, the coefficients of the matrix $\cA(b)^{-1}$ involved in $(\ref{systemNN})$ are smooth functions of $b$. This implies readily that the map $\Psi$ is also \textit{smooth}, which proves that the map $\overline{X}$ is. 
\end{proof}

\subsubsection{Decomposition of vector fields on $\B$}

For $W \in \X(\B)$ and $b \in \V$, $\overline{W}(b) \in \fm$, so that 
\begin{equation}\label{decompXbar}
\overline{W}(b) = \sum\limits_{k=1}^{N} w_{k}(b) \epsilon_{k}, 
\quad \textrm{where }
w_{k}(b):=\langle \overline{W}(b), \epsilon_{k} \rangle.
\end{equation}

\noindent
For $b \in \V$, applying the linear map $d_{e}\pi_{b}$ to both sides of $(\ref{decompXbar})$, we obtain that 
\begin{equation}\label{decompXb}
W(b) = \sum\limits_{k=1}^{N} w_{k}(b) \left( (d_{e}\pi_{b})(\epsilon_{k}) \right) = \sum\limits_{k=1}^{N} w_{k}(b) \epsilon_{k}^{*}(b). 
\end{equation}

\noindent
By Proposition $\ref{smoothXbar}$, the function $\overline{W}$ is smooth on $\V$. So, for all $1 \leq k \leq n$, the coordinate function $w_{k} : \V \lra \R$ is \textit{smooth}. Now, the function $\overline{W}$ is valued in the \textit{fixed} vector space $\fm$. Thus, for all $b \in \V$ and $\Delta \in T_{b}\B$,
\begin{equation}\label{diffXbar}
(d_{b}\overline{W})(\Delta) = 
\sum\limits_{k=1}^{N} \left( (d_{b}w_{k})(\Delta) \right) \epsilon_{k}. 
\end{equation}

\noindent
The preceding results are summarized as follows, where the \textit{smoothness} is an important point. 

\begin{prop}\label{propDecompVF}
Any $W \in \X(\B)$ is decomposed on $\V$ into a linear combination of the $(\epsilon_{k}^{*})_{k}$, with smooth coefficients which are the coordinate functions of the $\fm$-lift function $\overline{W}$ in the basis $(\epsilon_{k})_{k}$. 
\end{prop}

\subsection{The Levi-Civita connection}\label{subsecLCCofB}

Our aim here is to compute the Levi-Civita connection $\left( \nabla^{\B}_{X} Y \right)_{\beta}$, for arbitrary $X, Y \in \X(\B)$ and $\beta \in \B$. Since $\G/\K$ is a normal homogeneous space, we recall that for any $u, v \in \fm$, 
\begin{equation}\label{LCCfvfGK}
\left( \nabla^{\G/\K}_{\widetilde{u}} \widetilde{v} \right)_{o} = - \frac{1}{2} \left( d_{e}\psi \right) \left( \lb u, v \rb \right).
\end{equation}

\noindent
We give hereunder the scheme for deriving $\left( \nabla^{\B}_{X} Y \right)_{\beta}$ from $\left( \nabla^{\G/\K}_{\widetilde{u}} \widetilde{v} \right)_{o}$. 
\[ \xymatrix{
  \left( \nabla^{\B}_{X} Y \right)_{\beta} \ar@{<=}[r]^{(3)} & \left( \nabla^{\B}_{Z} W \right)_{b_{0}} \ar@{<=}[r]^{(2)} & 
\left( \nabla^{\B}_{\epsilon_{k}^{*}} \epsilon_{\ell}^{*} \right)_{b_{0}} \ar@{<=}[r]^{(1)} & \left( \nabla^{\B}_{\widetilde{\epsilon_{k}}} \widetilde{\epsilon_{\ell}} \right)_{o} 
  }
\]

\noindent\\
Arrow $(1)$ provides $\left( \nabla^{\B}_{\epsilon_{k}^{*}} \epsilon_{\ell}^{*} \right)_{b_{0}}$ in function of $\left( \nabla^{\B}_{\widetilde{\epsilon_{k}}} \widetilde{\epsilon_{\ell}} \right)_{o}$, which is in turn given by $(\ref{LCCfvfGK})$: see paragraph $\ref{paragraphLCCfvfB}$. Arrow $(2)$ describes how to deduce the connection at $b_{0}$, for any vector fields on $\V$, from that for fundamental ones: see paragraph $\ref{LCCorigin}$. Arrow $(3)$ means that the computation at any $\beta \in \B$ is reduced to that at $b_{0}$: see paragraph $\ref{intrinsicLCC}$, where the vector fields $Z$ and $W$ are precised. The final expression of $\left( \nabla^{\B}_{X} Y \right)_{\beta}$ is obtained in Theorem $\ref{finalLCCapplied}$. A key tool is the compatibility of the Levi-Civita connection with isometries, stated hereafter. 

\begin{lem}\label{lemNaturalityLCC}
Let $\phi : \left(\M, g \right) \lra \left(\M', g' \right)$ be an isometry. Then, for all $X, Y \in \X(\M')$ and $p \in \M$,
\begin{equation}\label{equNaturalityLCC}
\left( \nabla^{\M'}_{X} Y \right)_{\phi(p)} = \left( d_{p}\phi \right) \left( \left( \nabla^{\M}_{(\phi^{-1})_{*}X} ~(\phi^{-1})_{*}Y \right)_{p} \right), 
\end{equation}

\noindent
where, for example, $(\phi^{-1})_{*}X \in \X(\M)$ is the pushforward of $X$ by $\phi^{-1}$ i.e., for any $p \in \M$, 
\begin{equation*}
\left( (\phi^{-1})_{*}X \right)_{p} = \left( d_{\phi(p)} \phi^{-1} \right) \left( X(\phi(p)) \right). 
\end{equation*}
\end{lem}

\begin{proof}
This is a consequence of Koszul's formula.
\end{proof}

\subsubsection{Expression for fundamental vector fields}\label{paragraphLCCfvfB}

\begin{lem}\label{lemLCCfvfB}
For all $u, v \in \fm$, 
\begin{equation*}
\left( \nabla^{\B}_{u^{*}} v^{*} \right)_{b_{0}} = - \frac{1}{2} \left( d_{e}\pi_{0} \right) \left( \lb u, v \rb \right).
\end{equation*}
\end{lem}

\begin{proof}
First , $(\ref{linkFVFs})$ implies that $\left( \widehat{\pi}^{-1} \right)_{*} u^{*}=\widetilde{u}$. So, by Lemma $\ref{lemNaturalityLCC}$ applied to the isometry $\widehat{\pi_{0}}$,
\begin{equation*}
\left( \nabla^{\B}_{u^{*}} v^{*} \right)_{b_{0}} = 
\left( d_{o}\widehat{\pi_{0}} \right) \left( \left( \nabla^{\G/\K}_{\widetilde{u}} \widetilde{v} \right)_{o} \right) = 
- \frac{1}{2}\left( d_{o}\widehat{\pi_{0}} \right) \left( \left( d_{e}\psi \right) \left( \lb u, v \rb \right) \right), 
\end{equation*}

\noindent
where the last equality follows from $(\ref{LCCfvfGK})$. Since $\widehat{\pi_{0}} \circ \psi = \pi_{0}$, we conclude by the chain rule. 
\end{proof}

\subsubsection{Expression at the origin for arbitrary vector fields}\label{LCCorigin}

Here, for any $Z, W \in \X(\B)$, we derive $\left( \nabla^{\B}_{Z} W \right)_{b_{0}}$ from $\left( \nabla^{\B}_{\epsilon_{k}^{*}} \epsilon_{\ell}^{*} \right)_{b_{0}}$. In fact, we obtain an \textit{intrinsic} expression of $\left( \nabla^{\B}_{Z} W \right)_{b_{0}}$, i.e. \textit{independent} of any basis of $\fm$. Namely, we decompose $Z$ and $W$ into combinations of the $\left( \epsilon_{k}^{*} \right)_{k}$ and we apply the Leibniz rule to compute $\left( \nabla^{\B}_{Z} W \right)_{b_{0}}$ from these decompositions. Then, we factorize the resulting expression, by identifying therein the right-hand sides of $(\ref{decompXbar})$ and $(\ref{diffXbar})$, which provides hereafter an \textit{intrinsic} formula for the connection at $b_{0}$.

\begin{prop}\label{propIntrinsicOrigin}
Let $Z, W \in \X(\B)$. Then, 
\begin{equation}\label{equIntrinsicOrigin}
\left( \nabla^{\B}_{Z} W \right)_{b_{0}} = (d_{e}\pi_{0}) \left( \left(d_{b_{0}}\overline{W}\right)(Z(b_{0})) - \frac{1}{2} \lb \overline{Z}(b_{0}) , \overline{W}(b_{0}) \rb \right). 
\end{equation}
\end{prop}

\begin{proof}
First, $\left(\nabla^{\B}_{Z} W \right)_{b_{0}}$ can be computed by restricting $Z$ and $W$ to $\V$. Then, on $\V$, by Proposition $\ref{propDecompVF}$, $Z = \sum\limits_{k=1}^{N} z_{k} \epsilon_{k}^{*}$ and $W = \sum\limits_{\ell=1}^{N} w_{\ell}\epsilon_{\ell}^{*}$ where $z_{k}$ and $w_{\ell}$ are \textit{smooth functions} on $\V$. Now, by the Leibniz rule, for all $b \in \V$,  
\begin{equation}\label{LeibnizRule}
\left( \nabla^{\B}_{Z} W \right)_{b} = \left( \sum\limits_{\ell=1}^{N} \left( Z \cdot w_{\ell} \right) \epsilon_{\ell}^{*} + w_{\ell} \left( \nabla^{\B}_{Z} \epsilon_{\ell}^{*} \right) \right) (b) = \cD_{b}(Z, W) + \cE_{b}(Z, W), 
\end{equation}

\noindent
where
\begin{equation}\label{defDandE}
\cD_{b}(Z, W) := \left( \sum\limits_{\ell=1}^{N} \left( Z \cdot w_{\ell} \right) \epsilon_{\ell}^{*} \right)(b) \qquad \textrm{and} \qquad
\cE_{b}(Z, W) := \left( \sum\limits_{\ell=1}^{N} w_{\ell} \left( \nabla^{\B}_{Z} \epsilon_{\ell}^{*} \right) \right)(b). 
\end{equation}

\noindent
On the one hand, we apply $(\ref{diffXbar})$ to $W = \sum\limits_{\ell=1}^{N} w_{\ell}\epsilon_{\ell}^{*}$ with $\Delta = Z(b)$, and we obtain that
\begin{align}
\cD_{b}(Z, W) &= \left(d_{e}\pi_{0} \right) \left[ \sum\limits_{\ell=1}^{N} 
\left[ (d_{b}w_{\ell})(Z(b)) \right] \epsilon_{\ell} \right] \label{DofLCCext} \\
&= \left(d_{e}\pi_{0} \right) \left[ (d_{b}\overline{W})(Z(b)) \right]. \label{DofLCCint}
\end{align}

\noindent
On the other hand, 
\begin{equation*}
\cE_{b}(Z, W) = \sum\limits_{k, \ell=1}^{N} z_{k}(b) w_{\ell}(b) \left( \nabla^{\B}_{\epsilon_{k}^{*}} \epsilon_{\ell}^{*} \right)_{b}. 
\end{equation*}

\noindent
In particular, for $b = b_{0}$, by Lemma $\ref{lemLCCfvfB}$,  
\begin{align}
\cE_{b_{0}}(Z, W) &= - \frac{1}{2} (d_{e}\pi_{0}) \left( \sum\limits_{k, \ell=1}^{N} z_{k}(b_{0}) w_{\ell}(b_{0}) \lb \epsilon_{k}, \epsilon_{\ell} \rb \right) \label{EofLCCext} \\
&= - \frac{1}{2} (d_{e}\pi_{0}) \left( \lb \overline{Z}(b_{0}) , \overline{W}(b_{0}) \rb \right). \label{EofLCCint}  
\end{align}

\noindent
Indeed, by $(\ref{decompXbar})$, $\overline{Z}(b_{0}) = \sum\limits_{k}^{N} z_{k}(b_{0}) \epsilon_{k}$ and $\overline{W}(b_{0}) = \sum\limits_{\ell=1}^{N} w_{\ell}(b_{0}) \epsilon_{\ell}$. Finally, we prove $(\ref{equIntrinsicOrigin})$ by combining $(\ref{defDandE})$, $(\ref{DofLCCint})$ and $(\ref{EofLCCint})$. 
\end{proof}

\subsubsection{Expression at any point for arbitrary vector fields}\label{intrinsicLCC}

Here, for any $\beta \in \B$, we derive $\left( \nabla^{\B}_{X} Y \right)_{\beta}$ from Proposition $\ref{propIntrinsicOrigin}$. Let $\V_{\beta}$ be the open neighborhood of $\beta$ defined by 
\begin{equation*}
\V_{\beta} := Q_{\beta} \cdot \V. 
\end{equation*}

\noindent
Let $X_{\beta}$ and $Y_{\beta}$ be the vector fields on $\V$ which are the pushforwards of $X_{|\V_{\beta}}$ and $Y_{|\V_{\beta}}$ by the \textit{isometry} $\cL^{\B}_{Q_{\beta}^{-1}}$ i.e. 
\begin{equation*}
X_{\beta}(b) = \left(Q_{\beta}^{-1}\right) \cdot \left[ X(Q_{\beta} \cdot b) \right]
\quad \textrm{and} \quad
Y_{\beta}(b) = \left(Q_{\beta}^{-1}\right) \cdot \left[ Y(Q_{\beta} \cdot b) \right]
, \quad b \in \V.   
\end{equation*}

\noindent
Now, we introduce hereafter the maps $V_{\beta}$ and $\phi_{\beta}$ and the vector $U^{\beta}_{0} \in \g$ which are elements involved in the expression of $\left( \nabla^{\B}_{X} Y \right)_{\beta}$ in $(\ref{LCCofBusable})$ below. First, the map $V_{\beta} : \V \lra \g$ is defined by 
\begin{equation*}
V_{\beta}(b) := \cR_{Q_{b}^{-1}}\left[ \left(Y_{\beta}(b) \right)^{\sharp}_{Q_{b}} \right] \in \fm_{b}. 
\end{equation*}

\noindent
Thus, $V_{\beta}(b)$ is the $\fm_{b}$-lift of $Y_{\beta}(b)$. So, Proposition $\ref{propMliftFunction}$ implies that for all $b \in \V$, 
\begin{equation*}
\overline{Y_{\beta}}(b) = p_{b}^{-1}\left( V_{\beta}(b) \right). 
\end{equation*}

\noindent
Then, we define the vector $U^{\beta}_{0} \in \g$ by 
\begin{equation*}
U^{\beta}_{0} := \lb \overline{X_{\beta}}(b_{0}) , \overline{Y_{\beta}}(b_{0}) \rb. 
\end{equation*}

\noindent
Finally, we consider the linear map $\phi_{\beta} : \g \lra T_{\beta}\B$ defined by 
\begin{equation*}
\phi_{\beta}(u) = Q_{\beta} \cdot \left[ \left( d_{e}\pi_{0} \right) (u) \right]. 
\end{equation*}

\noindent
The proof of Theorem $\ref{finalLCCapplied}$ hereafter uses Proposition $\ref{PropDYbar}$ below, as an auxiliary result.

\begin{theo}\label{finalLCCapplied}
$(i)$ Let $X, Y \in \X(\B)$ and $\beta \in \B$. Then, 
\begin{equation}\label{LCCofBusable}
\left( \nabla^{\B}_{X} Y \right)_{\beta} = \left( d_{b_{0}} \left( \phi_{\beta} \circ V_{\beta} \right) \right)\left( X_{\beta}(b_{0}) \right) - \frac{1}{2} \phi_{\beta}\left( U^{\beta}_{0} \right) = \left( \at{\frac{d}{dt}}{t=0} \left( \phi_{\beta} \circ V_{\beta} \right)(\gamma(t)) \right) - \frac{1}{2} \phi_{\beta}\left( U^{\beta}_{0} \right), 
\end{equation}

\noindent
where $\gamma$ is the geodesic in $\B$ through $b_{0}$ in the direction $X_{\beta}(b_{0})$, given by $(\ref{geodesicB})$.

\noindent\\
$(ii)$ If the Bracket Condition $(\ref{BracketCondB})$ holds, then $U^{\beta}_{0} \in \ker(d_{e}\pi_{0})$, i.e. the second term of $(\ref{LCCofBusable})$ vanishes. 
\end{theo}

\begin{proof}
We can compute $\left( \nabla^{\B}_{X} Y \right)_{\beta}$ by restricting $X$ and $Y$ to $\V_{\beta}$. By Lemma $\ref{lemNaturalityLCC}$ applied with the isometry $\cL^{\B}_{Q_{\beta}}$, 
\begin{equation}\label{LCCoriginToAny}
\left( \nabla^{\B}_{X} Y \right)_{\beta} = Q_{\beta} \cdot \left[ \left( \nabla^{\B}_{X_{\beta}} Y_{\beta} \right)_{b_{0}} \right]. 
\end{equation}

\noindent
By combining $(\ref{LCCoriginToAny})$ with Proposition $\ref{propIntrinsicOrigin}$ applied to $Z=X_{\beta}$ and $W=Y_{\beta}$, we obtain that 
\begin{align}\label{intrinsicGenLCC}
\left( \nabla^{\B}_{X} Y \right)_{\beta} &= Q_{\beta} \cdot \left[ (d_{e}\pi_{0}) \left( \left(d_{b_{0}}\overline{Y_{\beta}}\right) \left(X_{\beta}(b_{0}) \right) - \frac{1}{2} \lb \overline{X_{\beta}}(b_{0}) , \overline{Y_{\beta}}(b_{0}) \rb \right) \right] \\
&= \phi_{\beta} \left( \left(d_{b_{0}}\overline{Y_{\beta}}\right) \left(X_{\beta}(b_{0}) \right) - \frac{1}{2} U^{\beta}_{0} \right) \\
&= \phi_{\beta} \left( \left(d_{b_{0}}V_{\beta}\right) \left(X_{\beta}(b_{0}) \right) - \frac{1}{2} U^{\beta}_{0} \right), 
\end{align}

\noindent
where the last equality follows from Proposition $\ref{PropDYbar}$ below. Finally, since $\phi_{\beta}$ is a linear map,
\begin{equation*}
\left( \nabla^{\B}_{X} Y \right)_{\beta} = 
\left( d_{b_{0}} \left( \phi_{\beta} \circ V_{\beta} \right) \right)\left( X_{\beta}(b_{0}) \right) - \frac{1}{2} \phi_{\beta}\left( U^{\beta}_{0} \right). 
\end{equation*}

\end{proof}

\begin{prop}\label{PropDYbar}
The map $V_{\beta} : \V \lra \g$ is smooth and for any $\Delta \in T_{b_{0}}\B$,
\begin{equation}\label{EqDYbar}
\left(d_{b_{0}}\overline{Y_{\beta}}\right)(\Delta) = \left( d_{b_{0}}V_{\beta} \right) \left( \Delta \right).
\end{equation}

\end{prop}

\begin{proof}
See Appendix. 
\end{proof}

\subsection{The sectional curvature}

The following Lemma states that the sectional curvature is invariant under any isometry.
  
\begin{lem}\label{lemNaturalityCurva}
Let $\phi : \left(\M, g \right) \lra \left(\M', g' \right)$ be an isometry. Then, for all $Z, T \in \X(\M')$,
\begin{equation}\label{eqNaturalityCurva}
K^{\M'}_{\phi(p)}(Z, T) = K^{\M}_{p}\left( \left(\phi^{-1}\right)_{*}Z , \left(\phi^{-1}\right)_{*}T \right),
\end{equation}

\noindent
where $K^{\M}$ and $K^{\M'}$ denote the sectional curvatures on $\M$ and $\M'$.
\end{lem}

\begin{proof}
This follows readily from Lemma \ref{lemNaturalityLCC}. 
\end{proof}

\noindent
Since $\G / \K$ is a normal homogeneous space, we recall that 
\begin{equation}\label{recallCurvaNormal}
K_{o}(\widetilde{u} , \widetilde{v}) = \left\| \lb u,v \rb_{\fk} \right\|^{2} + \frac{1}{4}\left\| \lb u,v \rb_{\fm} \right\|^{2}. 
\end{equation}

\noindent
We deduce hereafter the sectional curvature at any point of $\B$ and for any vector field on $\B$.

\begin{prop}\label{propFinalCurva}
$(i)$ Let $X, Y \in \X(\B)$. For any $\beta \in \B$, consider $U^{\beta}_{0} \in \g$ defined in $(\ref{U0beta})$ hereabove. Then, 
\begin{equation}\label{curvaB}
K^{\B}_{\beta}\left(X, Y\right) = \left\| \left(U^{\beta}_{0} \right)_{\fk} \right\|^{2} + \frac{1}{4} \left\| \left(U^{\beta}_{0} \right)_{\fm} \right\|^{2}.
\end{equation}

\noindent\\
$(ii)$ If the Bracket Condition $(\ref{BracketCondB})$ holds, then $U^{\beta}_{0} \in \fk$. So, 
\begin{equation}\label{curvaBBracket}
K^{\B}_{\beta}\left(X, Y\right) = \left\| U^{\beta}_{0} \right\|^{2}.
\end{equation}

\end{prop}

\begin{proof}
We have that 
\begin{align}
K^{\B}_{\beta}\left(X, Y\right) &= K^{\B}_{b_{0}} \left( X_{\beta} , Y_{\beta} \right)\label{L1} \\
&= K^{\G/\K}_{o} \left( \left( \widehat{\pi_{0}}^{-1} \right)_{*}X_{\beta} ~,~ \left( \widehat{\pi_{0}}^{-1} \right)_{*}Y_{\beta} \right) \label{L2}\\
& = K^{\G/\K}_{o} \left( \widetilde{\overline{X_{\beta}} \left(b_{0}\right)} ~,~ \widetilde{\overline{Y_{\beta}} \left(b_{0}\right)} \right) \label{L3}. 
\end{align}

\noindent
$(\ref{L1})$ and $(\ref{L2})$ follow from Lemma $\ref{lemNaturalityCurva}$, applied to $\cL^{\B}_{Q_{\beta}}$ and to $\widehat{\pi_{0}}$. For $(\ref{L3})$, it is enough to prove that 
\begin{equation}\label{conditionForL3}
\left( \left( \widehat{\pi_{0}}^{-1} \right)_{*}X_{\beta} \right) (o) = \left( \widetilde{\overline{X_{\beta}} \left(b_{0}\right)} \right) (o)
\quad \textrm{i.e.} \quad 
\left(d_{b_{0}}\left(\widehat{\pi_{0}}^{-1}\right) \right)(X_{\beta}(b_{0})) = \left( d_{e}\psi \right) \left( \overline{X_{\beta}} \left(b_{0}\right) \right).
\end{equation}

\noindent
Indeed, the sectional curvature at any point depends only on the values of the vector fields at this point. Now, $\left(d_{o}\widehat{\pi_{0}}\right) \circ \left(\left( d_{e}\psi \right)_{|\fm}\right) = \left( d_{e}\pi_{0} \right)_{|\fm}$. Thus, 
\begin{equation}\label{chainRule}
d_{b_{0}}\left(\widehat{\pi_{0}}^{-1}\right) = \left(d_{o}\widehat{\pi_{0}}\right)^{-1} =
\left( d_{e}\psi \right)_{|\fm} \circ \left( d_{e}\pi_{0} \right)_{|\fm}^{-1}. 
\end{equation}

\noindent
Then, $(\ref{conditionForL3})$ is obtained by evaluating both members of $(\ref{chainRule})$ at $X_{\beta}\left(b_{0}\right)$. Finally, by $(\ref{recallCurvaNormal})$, $(i)$ holds. 
\end{proof}

\section{Geometry of flag manifolds}\label{sect4}

Recall that, for a type $\I=(q_{i})_{1 \leq i \leq r}$, the compact connected Lie group $SO(n)$ acts smoothly on the product manifold $G^{\I}:=\prod G_{q_{i}}$ by 
\begin{equation}\label{actionOfOrtho}
(Q, \cP) \lmp Q \cdot \cP := \left( QP_{1}Q^{T}, ..., QP_{r}Q^{T}\right), \quad Q \in SO(n), \cP = (P_{i})_{i} \in G^{\I}.  
\end{equation}

\noindent
We prove in Corollary $\ref{corrPiiSurjective}$ below that $\F$ is the orbit of the standard flag $\left( P_{0}^{i}\right)_{1 \leq i \leq r}$ defined in $(\ref{DefStandardFlag})$. Now, the  isotropy group of $\cP_{0}$ is $SO(\I):=SO(n) \cap O(\I)$. Thus, $\F$ is endowed with the metric $g^{\cF}$ such that the map $\widehat{\pi}_{0}^{\I} : \left( SO(n)/SO(\I), \overline{g}^{O} \right) \lra \left( \F, g^{\cF} \right)$, induced by $\piio$, is an \textit{isometry}. So, we may apply the results of Section $\ref{sect3}$ to $\F$, which is the object of this Section.

\subsection{Basic geometry of flag manifolds}

\subsubsection{The main orbital map}\label{mainOMflag}

Given a type $\I=(q_{i})_{1 \leq i \leq r}$, we partition any $n \times n$ matrix $M$ into blocks of its columns: 
\begin{equation*}
M = \begin{pmatrix}
M^{(1)} & ... & M^{(i)} & ... & M^{(r)} 
\end{pmatrix}. 
\end{equation*}

\noindent
where $M^{(i)} \in \R^{n \times q_{i}}$ is called the $q_{i}$-block of $M$.

\begin{defi}\label{defiGenerateFlag}
Let $\cP$ be a flag of type $\I=(q_{i})$. For $Q \in O(n)$, we say that its columns generate $\cP=(P_{i})_{i}$ when for all $1 \leq i \leq r$, the columns of $Q^{(i)}$ form a $q_{i}$-frame generating the range of $P_{i}$. 
\end{defi}

\begin{lem}\label{FiberPii}
Let $\cP$ be a flag of type $\I=(q_{i})_{1 \leq i \leq r}$ and $Q \in O(n)$. Then, the columns of $Q$ generate the flag $\cP$ iff 
\begin{equation}\label{generateFlag}
\pii(Q) = \cP, 
\end{equation}

\noindent
where the map $\pii : O(n) \lra \F$ is defined by $\pii(Q)=Q \cdot \cP_{0}$, whose restriction to $SO(n)$ is $\piio$. 
\end{lem}

\begin{proof}
See Appendix.
\end{proof}

\begin{corr}\label{corrPiiSurjective}
The map $\pii : O(n) \lra \F$ and its restriction $\piio$ to $SO(n)$ are surjective. 
\end{corr}

\begin{proof}
The assertion on $\pii$ follows readily from $(\ref{generateFlag})$. Now, we prove the statement on $\piio$. Any equivalence class $QO(\I)$ in $O(n)/O(\I)$ contains an element $Q'$ of $SO(n)$, so that $\pii(Q')=\pii(Q)$. So, $\piio$ is also surjective.
\end{proof}

\subsubsection{Vertical and horizontal spaces}

Since $SO(n)$ is a matrix Lie group, the Lie bracket in its Lie algebra $\son$ is the commutator of matrices, i.e. for any $A, B \in \son$, $\lb A, B \rb = AB-BA$.

\begin{lem}\label{diffPi}
$(i)$ Let $\cP=(P_{i})_{i} \in \F$, $Q \in SO(n)$ and $Q\Omega \in T_{Q}SO(n)$ with $\Omega \in \mathfrak{so}(n)$. Then,  
\begin{equation*}
\left(d_{Q}\pi_{\cP} \right) \left(Q\Omega \right) = \left( Q\left[ \Omega, P_{i}\right] Q^{T} \right)_{i}. 
\end{equation*}

\noindent
We recall that $\piio$ is defined as the main orbital map $\pi_{\cP_{0}}$. Then, 

\smallskip

\noindent
$(ii)$ The vertical space at $Q \in SO(n)$ wrt $\piio$ is 
\begin{equation}\label{verSpaceFlag}
V_{Q}^{\piio} := \ker \left( d_{Q}\piio \right) = \left\{ Q \mathrm{Diag}\left[ \Omega_{11}, ..., \Omega_{ii}, ..., \Omega_{rr}\right] : \Omega_{ii} \in \mathfrak{so}(q_{i}), ~ 1 \leq i \leq r \right\}. 
\end{equation}

\noindent
$(iii)$ The horizontal space at $Q \in SO(n)$ wrt $\piio$ is 
\begin{equation}\label{horSpaceFlag}
H_{Q}^{\piio} := \left( V_{Q}^{\piio} \right)^{\perp} = \left\{ Q\Omega : \Omega \in \mathfrak{so}(n), ~\Omega_{ii} = 0, ~ 1 \leq i \leq r \right\}. 
\end{equation}

\end{lem}

\begin{proof}
Let $\gamma : [0,1] \longrightarrow SO(n)$ be a smooth curve with $\gamma(0)=Q$ and $\dot{\gamma}(0)=Q\Omega$. Then, 
\begin{equation*}
\left(d_{Q}\pi_{\cP} \right) \left(Q\Omega \right) = \left( \at{\frac{d}{dt}}{t=0} \gamma(t) P_{i} \gamma(t)^{T} \right)_{i} = 
\left( \left(Q\Omega \right)P_{i}Q^{T} + QP_{i}\left(Q\Omega \right)^{T} \right)_{i} = \left( Q\left[ \Omega, P_{i}\right] Q^{T} \right)_{i}.
\end{equation*}

\noindent
Indeed, $(Q\Omega)^{T} = - \Omega Q^{T}$, since $\Omega \in \son$. This proves $(i)$. We deduce that
\begin{equation*}
V_{Q}^{\piio} = \left\{ Q\Omega \in T_{Q}O(d) : \left[ \Omega, P_{0}^{i}\right]=0, ~ 1 \leq i \leq r \right\}.
\end{equation*}

\noindent
Then, the computation of $\left[ \Omega, P_{0}^{i}\right]$ implies that $(ii)$ holds. Finally, $(iii)$ follows readily from $(ii)$. 
\end{proof}

\subsubsection{Horizontal lifts of vector fields}

\begin{lem}\label{PrelimHorLift}
For all $\Omega \in \son$, let $\Omega_{\fm}$ be its horizontal part. Then, 
\begin{equation}\label{horPii}
\Omega_{\fm} = \frac{1}{2} \sum\limits_{i=1}^{r} \lb \lb \Omega, P_{0}^{i} \rb, P_{0}^{i} \rb. 
\end{equation}
\end{lem}

\begin{proof}
See Appendix. 
\end{proof}

\noindent
We deduce the following result, which provides the horizontal lifts of vector fields on $\F$.

\begin{prop}\label{horLiftFlag}
Let $\cP \in \F$ and $\Delta \in T_{\cP}\F$. Then, for $Q \in \left( \piio \right)^{-1}(\cP)$, the horizontal lift at $Q$ of $\Delta$ wrt $\piio$ is
\begin{equation}\label{flagHorLift}
\Delta^{\sharp}_{Q} = \frac{1}{2} \left( \sum\limits_{i=1}^{r} \lb \Delta_{i}, P_{i} \rb \right) Q. 
\end{equation}

\end{prop}

\begin{proof}
There exists a unique $\widehat{\Omega} \in \fm$ such that $\Delta^{\sharp}_{Q} = Q\widehat{\Omega}$. Thus, $\Delta = \left( d_{Q}\piio \right)\left( Q\widehat{\Omega} \right) = \left( Q \lb \widehat{\Omega}, P_{0}^{i} \rb Q^{T} \right)_{i}$. Then, considering the right-hand side of $(\ref{flagHorLift})$,
\begin{equation*}
\frac{1}{2} \left( \sum\limits_{i=1}^{r} \lb \Delta_{i}, P_{i} \rb \right) Q = \frac{1}{2} \left( \sum\limits_{i=1}^{r} \lb Q \lb \widehat{\Omega}, P_{0}^{i} \rb Q^{T}, QP_{0}^{i}Q^{T} \rb \right) Q = Q \left( \frac{1}{2} \sum\limits_{i=1}^{r} \lb \lb \widehat{\Omega}, P_{0}^{i} \rb, P_{0}^{i} \rb \right).
\end{equation*}

\noindent
Now, $(\ref{horPii})$ implies that 
\begin{equation*}
Q \left( \frac{1}{2} \sum\limits_{i=1}^{r} \lb \lb \widehat{\Omega}, P_{0}^{i} \rb, P_{0}^{i} \rb \right) = Q \widehat{\Omega}_{\fm} = Q\widehat{\Omega} = \Delta^{\sharp}_{Q}.
\end{equation*}

\noindent
Hereabove, the second equality holds because $\widehat{\Omega}$ lies in $\fm$.  
\end{proof}

\subsubsection{Metric, geodesics and exponential map}

By Lemma $\ref{structureMainOM}$, the main orbital map $\piio : \left( SO(n), g^{O} \right) \lra \left( \F, g^{\cF} \right)$ is a \textit{Riemannian submersion}. We deduce readily from the expression of horizontal lifts given in Proposition $\ref{horLiftFlag}$ that for all $\cP \in \F$, 
\begin{equation*}
g^{\cF}_{\cP} \left( \Delta, \Delta' \right) = \frac{1}{4} \sum\limits_{i, j=1}^{r} \left\langle \lb \Delta_{i}, P_{i} \rb , \lb \Delta'_{j}, P_{j} \rb \right\rangle, 
\quad \Delta, \Delta' \in T_{\cP}\F. 
\end{equation*}

\begin{rem}\label{metricDifferYeKe}
The metric $g^{\cF}$ on $\F$ does not coincide with the restriction to $\F$ of the product metric on $G^{\I}$ which is considered in \cite{Ye Wong and Lim 2022}. This is proven by elementary calculations. Intuitively, in $\so$, the horizontal space wrt $\piio$ does not agree with those wrt the main orbital maps of the $\left(G_{q_{i}}\right)_{1 \leq i \leq r}$. This follows from the mutual orthogonality relations in $\F$. 
\end{rem}

\begin{prop}\label{GeoExpMap}
Let $\mathcal{P} \in \F$ and $\Delta \in T_{\mathcal{P}}\F$. Then, the geodesic in $\left( \F, g^{\cF} \right)$ through $\mathcal{P}$ in the direction $\Delta$ is 
\begin{equation}\label{geoFlag}
t \lmp \piio \left( Q \exp_{m} \left( t Q^{T}\widetilde{\Omega}Q \right) \right)
\quad \textrm{where} \quad 
\widetilde{\Omega} = \frac{1}{2} \sum\limits_{i=1}^{r} \lb \Delta_{i}, P_{i} \rb
\end{equation}

\noindent
and $\exp_{m}$ is the matrix exponential. So, the Riemannian exponential map on $\F$ is defined by
\begin{equation}\label{ExpMap}
\mathrm{Exp}_{\mathcal{P}}^{\mathcal{F}^{\mathrm{I}}}\left( \Delta \right) = 
\left( \exp_{m}\left( \widetilde{\Omega} \right) P_{i} \exp_{m}\left( - \widetilde{\Omega}\right)\right)_{i}.
\end{equation}

\end{prop}

\begin{proof}
We apply $(\ref{geodesicB})$. Here, $\exp_{\so}=\exp_{m}$ and by $(\ref{flagHorLift})$, $\cL_{Q^{-1}} \left[\Delta^{\sharp}_{Q}\right] = Q^{T}\widetilde{\Omega}Q$. So,  
\begin{equation*}
\mathrm{Exp}_{\mathcal{P}}^{\mathcal{F}^{\mathrm{I}}}\left( \Delta \right) =
\piio \left( Q \exp_{m} \left( Q^{T}\widetilde{\Omega}Q \right) \right) =
\left( \exp_{m}\left( \widetilde{\Omega} \right) P_{i} \exp_{m}\left( - \widetilde{\Omega}\right)\right)_{i}. 
\end{equation*}

\noindent
Hereabove, the second equality follows from the properties of $\exp_{m}$ and the definition of $\piio$.
\end{proof}

\begin{rem}
Given a Riemannian submersion $\pi : \G \lra \B$, the expression of the logarithm map of $\B$ can not be deduced in general from that of $\G$. Thus, we do not obtain a closed form for the logarithm in $\F$. However, algorithms for approaching it exist: See \cite{Le Brigant Arnaudon and Barbaresco 2017} or \cite{Thanwerdas and Pennec 2021} for the general case and \cite{Nguyen 2022} for the specific case of $\F$. 
\end{rem}

\subsection{Levi-Civita connection and curvature of flag manifolds}

Let $X, Y \in \X(\F)$ and $\cP=(P_{i})_{i} \in \F$. For all $1 \leq i \leq r$, set $X_{i}(\cP) := \left( X(\cP) \right)_{i} \in T_{P_{i}}G_{i}$. Set 
\begin{equation*}
\Omega_{\cP}\left(X, Y \right) := \sum\limits_{i, j=1}^{N} \lb \lb X_{i}(\cP) , P_{i} \rb , \lb Y_{j}(\cP) , P_{j} \rb  \rb \in \son. 
\end{equation*}

\noindent
In order to derive the Levi-Civita connection of $\F$ from Theorem $\ref{finalLCCapplied}$, we perform preliminary computations. By Lemma $\ref{diffPi}$, with the notations of Theorem $\ref{finalLCCapplied}$, for any $\beta \in \F$ and $\Omega \in \son$,
\begin{equation*}
\phi_{\beta}(\Omega) := Q_{\beta} \cdot \left[ \left( d_{e}\pii \right)(\Omega) \right] = \left( Q_{\beta} \lb \Omega , P_{0}^{i} \rb Q_{\beta}^{T} \right)_{i}.
\end{equation*}

\noindent
On the one hand, by Proposition $\ref{horLiftFlag}$, for all $\cP \in \V$,
\begin{equation*}
V_{\beta}(\cP) := \cR_{Q_{\cP}^{-1}}\left[ \left(Y_{\beta}(\cP) \right)^{\sharp}_{Q_{\cP}} \right] 
= \frac{1}{2} \sum\limits_{j=1}^{r} \lb \left( Y_{\beta}(\cP) \right)_{j}, \cP_{j} \rb. 
\end{equation*}

\noindent
Therefore, 
\begin{align}
\left( \phi_{\beta} \circ V_{\beta} \right)(\cP) &= \left( Q_{\beta} \lb V_{\beta}(\cP) , P_{0}^{i} \rb Q_{\beta}^{T} \right)_{i} \\
&= \left( Q_{\beta} \lb \frac{1}{2} \sum\limits_{j=1}^{r} \lb \left( Y_{\beta}(\cP) \right)_{j}, \cP_{j} \rb ~,~ P_{0}^{i} \rb Q_{\beta}^{T} \right)_{i} \\
&= \frac{1}{2} \left( \lb \sum\limits_{j=1}^{r} \lb Q_{\beta} \left( Y_{\beta}(\cP) \right)_{j} Q_{\beta}^{T} , Q_{\beta} \cP_{j} Q_{\beta}^{T} \rb ~,~ Q_{\beta}P_{0}^{i}Q_{\beta}^{T} \rb \right)_{i} \\
\label{phiRondVFlag} &= \frac{1}{2} \left( \lb \sum\limits_{j=1}^{r} \lb Y_{j}\left( Q_{\beta} \cdot \cP \right) , \left( Q_{\beta} \cdot \cP \right)_{j} \rb ~,~ \beta_{i} \rb \right)_{i}.
\end{align}

\noindent
The last line hereabove holds because, by definition of $Y_{\beta}$, for all $j$, 
\begin{equation*}
Q_{\beta} \left( Y_{\beta}(\cP) \right)_{j} Q_{\beta}^{T} = \left( Q_{\beta} \cdot Y_{\beta}(\cP) \right)_{j} 
= \left( Q_{\beta} \cdot \left(Q_{\beta}^{-1}\right) \cdot \left[ Y(Q_{\beta} \cdot \cP) \right] \right)_{j} = \left( Y(Q_{\beta} \cdot \cP) \right)_{j}
=: Y_{j}\left( Q_{\beta} \cdot \cP \right).
\end{equation*} 

\noindent
On the other hand, for $X, Y \in \X(\F)$ and $\beta=(\beta_{i})_{i} \in \F$, 
\begin{equation}\label{UbetaFlag}
U^{\beta}_{0} = \frac{1}{4} Q_{\beta}^{T} \Omega_{\beta}(X, Y) Q_{\beta}
\quad \textrm{and} \quad 
\phi_{\beta} \left( U^{\beta}_{0} \right) = \frac{1}{4} \left( \lb \Omega_{\beta}(X, Y) , \beta_{i} \rb \right)_{i}. 
\end{equation}

\noindent
We deduce hereafter an expression of the Levi-Civita connection of $\F$ that is \textit{independent} of $Q_{\beta}$.

\begin{theo}\label{StatementTheoLCCflag}
Let $X, Y \in \X(\F)$. Then, for any $\beta=(\beta_{i})_{i} \in \F$,  
\begin{equation}\label{LCCFlagTheo}
\left( \nabla^{\F}_{X} Y \right)_{\beta} = \left( \lb \Omega'(\beta) - \frac{1}{8} \Omega_{\beta}(X, Y) ~,~  \beta_{i} \rb \right)_{i} 
\end{equation}

\noindent
where 
\begin{equation*}
\Omega'(\beta) = \frac{1}{2} \at{\frac{d}{dt}}{t=0} \left( \sum\limits_{j=1}^{r} \lb Y_{j}\left( \Gamma(t) \right) , \left( \Gamma(t) \right)_{j} \rb \right) \in \son
\end{equation*}

\noindent
and $\Gamma$ is the geodesic in $\F$ through $\beta$ in the direction $X(\beta)$.
\end{theo}

\begin{proof}
By $(i)$ of Theorem $\ref{finalLCCapplied}$,
\begin{equation}\label{LCCFlagLem}
\left( \nabla^{\F}_{X} Y \right)_{\beta} = \left( \at{\frac{d}{dt}}{t=0} \left( \phi_{\beta} \circ V_{\beta} \right)(\gamma(t)) \right) - \frac{1}{2} \phi_{\beta}\left( U^{\beta}_{0} \right), 
\end{equation}

\noindent
where $\gamma$ is the geodesic in $\F$ through $\cP_{0}$ in the direction $X_{\beta}\left(\cP_{0}\right)$. By $(\ref{phiRondVFlag})$, 
\begin{equation*}
\left( \phi_{\beta} \circ V_{\beta} \right) \left(\gamma(t) \right) = 
\frac{1}{2} \left( \lb \sum\limits_{j=1}^{r} \lb Y_{j}\left( Q_{\beta} \cdot \gamma(t) \right) , \left( Q_{\beta} \cdot \gamma(t) \right)_{j} \rb ~,~ \beta_{i} \rb \right)_{i}.
\end{equation*}

\noindent
 Now, $\cL^{\B}_{Q_{\beta}}$ is an isometry of $\F$. So, $Q_{\beta} \cdot \gamma(t) = \Gamma(t)$, where $\Gamma$ is the geodesic in $\F$ through $Q_{\beta} \cdot \cP_{0} = \beta$ in the direction $Q_{\beta} \cdot \left[ X_{\beta}\left(\cP_{0}\right) \right] = X(Q_{\beta} \cdot \cP_{0}) = X(\beta)$. Finally, we deduce $(\ref{LCCFlagTheo})$ by combining $(\ref{LCCFlagLem})$ and $(\ref{UbetaFlag})$. 
\end{proof}

\noindent\\
Finally, the sectional curvature of $\F$ is derived from the results of Section $\ref{sect3}$ as follows.

\begin{theo}
Let $X, Y \in \X(\F)$. Then, for any $\beta=(\beta_{i})_{i} \in \F$,  
\begin{equation}\label{CurvaFlagTheo}
K^{\F}_{\beta}(X, Y) = \frac{1}{16} \left( \left\| \Omega_{\beta}(X, Y) \right\|^{2} - \frac{3}{16} \left\| \sum\limits_{i=1}^{r} \lb \lb \Omega_{\beta}(X, Y) , \beta_{i} \rb , \beta_{i} \rb \right\|^{2}
\right)
\end{equation}
\end{theo}

\begin{proof}
By Proposition $\ref{propFinalCurva}$ and the Pythagorean theorem, 
\begin{equation*}
K^{\F}_{\beta}(X, Y) = \left\| U^{\beta}_{0} \right\|^{2} - \frac{3}{4} \left\| \left(U^{\beta}_{0} \right)_{\fm} \right\|^{2}
\end{equation*}

\noindent
Recall from $(\ref{UbetaFlag})$ that for $\F$, $U^{\beta}_{0} = \frac{1}{4} Q_{\beta}^{T} \Omega_{\beta}(X, Y) Q_{\beta}$. Then, we conclude by Lemma $\ref{PrelimHorLift}$, which provides the expression of $\left(U^{\beta}_{0} \right)_{\fm}$ in function of $U^{\beta}_{0}$. 
\end{proof}

\subsection{The case of Grassmannians}

As pointed out in Section $\ref{sect2}$, the Bracket Condition holds for $\Gr$, since $\Gr$ is additionally a \textit{symmetric space}. This result can also be checked by direct computations from the explicit expressions of the spaces $\fk$ and $\fm$ in the case of $\Gr$ given in \cite{Bendokat Zimmermann and Absil 2020}. As for $\F$, we derive the expression of the Levi-Civita connection of $\Gr$ by applying Theorem $\ref{finalLCCapplied}$. By results of \cite{Bendokat Zimmermann and Absil 2020}, with the notations of Theorem $\ref{finalLCCapplied}$, for any $\beta \in \Gr$ and $P \in \V$,
\begin{equation*}
\left( \phi_{\beta} \circ V_{\beta} \right)(P) = \lb \lb Y( Q_{\beta} \cdot P) , Q_{\beta} \cdot P \rb, \beta \rb.
\end{equation*}

\noindent
Thus, by applying $(ii)$ of Theorem $\ref{finalLCCapplied}$, we obtain hereafter the Levi-Civita connection of $\Gr$.

\begin{theo}
Let $X, Y \in \X(\Gr)$. Then, for any $\beta \in \Gr$,  
\begin{equation}\label{LCCGrassTheo}
\left( \nabla^{\Gr}_{X} Y \right)_{\beta} = \lb \Omega'(\beta) , \beta \rb 
\quad \textrm{where} \quad \Omega'(\beta) = \at{\frac{d}{dt}}{t=0} \lb Y(\Gamma(t)) , \Gamma(t) \rb, 
\end{equation}

\noindent
and $\Gamma$ is the geodesic in $\Gr$ through $\beta$, in the direction $X(\beta)$.
\end{theo}

\noindent\\
We readily check that the sectional curvature of $\Gr$ derived from our method coincides with that obtained in \cite{Bendokat Zimmermann and Absil 2020}.

\subsection{Explicit local section of the map $\pii$}

Since a closed form for the  Riemannian Logarithm on $\F$ is not available, we can not derive explicit local sections of $\pii$  by applying Corollary $\ref{corLocSection}$ with $\B=\F$. We overcome this lack by using the embedding of $\F$ into $G^{\I}$. Indeed, closed forms for the cut locus and the Logarithm map on $\Gr$, are available: See $(\ref{CutGrass})$ and $(\ref{LogGrass})$ below.

\subsubsection{Preliminaries}

For $1 \leq q \leq n$, let $\St_{q}=\left\{ U \in \R^{n \times q} : U^{T}U = I_{q} \right\}$ be the Stiefel manifold of $q$-frames of $\R^{n}$. Let $\pi^{SG}_{q} : \St_{q} \lra \Gr$ be the \textit{surjective} map which associates to a $q$-frame $U$, the projection onto the linear subspace that it generates. Namely, $\pi^{SG}_{q}\left(U\right)=UU^{T}$. Then, the \textit{cut locus} of any $P=UU^{T} \in \Gr$, with $U \in \St_{q}$, is  
\begin{equation}\label{CutGrass}
\Ct(P) = \left\{ R=YY^{T} \in \Gr : \rk\left( U^{T}Y \right)<q \right\}. 
\end{equation}

\noindent
In other words, $R \in \Ct(P)$ iff at least one \textit{principal angle} between $P$ and $R$ is equal to $\frac{\pi}{2}$. Furthermore, the Logarithm map on $\Gr$ is provided by Theorem $3.3$ in \cite{Batzies Huper Machado and Silva Leite 2015}: For any $P \in \Gr$ and $R \in G_{q} \setminus \Ct(P)$, 
\begin{equation}\label{LogGrass}
\Log_{P}^{\Gr} \left(R\right) = \lb \Omega_{P}(R), P \rb
\quad \textrm{with} \quad \Omega_{P}(R) = \frac{1}{2} \log_{m} \left( (I_{n}-2R)(I_{n}-2P) \right) \in \son.  
\end{equation}

\noindent
In the sequel, $\St_{q}$ is endowed with the Euclidean metric $g^{\St}$ defined by
\begin{equation*}
g^{\St}_{U}\left( \cD_{1}, \cD_{2} \right)=\tr \left( \cD_{1}^{T} \cD_{2} \right), \quad U \in \St_{q} \textrm{~and~} \cD_{1}, \cD_{2} \in T_{U}\St_{q}. 
\end{equation*}

\begin{lem}\label{piSG}
$(i)$ For $P \in \Gr$, $U \in \left( \pi^{SG}_{q} \right)^{-1}\left(P\right)$ and $\Delta \in T_{P}\Gr$, the horizontal lift of $\Delta$ at $U$ wrt $\left(\pi^{SG}_{q}, g^{\St}\right)$ is 
\begin{equation}\label{horLiftSG}
\Delta^{\sharp}_{U}=\Delta U.  
\end{equation} 

\noindent
$(ii)$ $\pi^{SG}_{q}$ is a Riemannian submersion from $\left( \St_{q}, g^{\St} \right)$ onto $\left( \Gr, g^{\Gr} \right)$. 
\end{lem} 

\begin{proof}
See Appendix. 
\end{proof}

\noindent
Since $\St_{q}$ is complete and $\Gr$ connected, Lemma $\ref{lemLocSection}$ provides local sections of $\pi^{SG}_{q}$ based on holonomy maps between fibers under $\pi_{q}^{SG}$. Recall that for $P \in \Gr$, $R \in G_{q} \setminus \Ct(P)$ and $U \in \left( \pi^{SG}_{q} \right)^{-1}(P)$, 
\begin{equation}\label{holonomyFormulaSG}
\cH^{\pi^{SG}_{q}}_{(P, R)}(U) = \Exp^{\St}_{U} \left( \left( \Log^{\Gr}_{P}(R) \right)^{\sharp}_{U} \right).
\end{equation}

\noindent
The exponential map on $\left( \St_{q}, g^{\St} \right)$ is provided by Section $2.2.2.$ in \cite{Edelman Arias and Smith 1998}: for $U \in \St_{q}$ and $\cD \in T_{U}\St_{q}$, 
\begin{equation}\label{ExpSt}
\Exp_{U}^{\St_{q}}\left( \cD \right) = V \left( \exp_{m} \begin{pmatrix}
A & -C \\
I_{q} & A
\end{pmatrix}
\right) \begin{pmatrix}
I_{q} \\ 0_{q} 
\end{pmatrix} \exp_{m} \left( - A \right),
\end{equation}

\noindent
where 
\begin{equation*}
V = \begin{pmatrix}
U & \cD 
\end{pmatrix} \in \R^{n \times 2q}
\quad , \quad A=U^{T}\cD \quad , \quad C=\cD^{T}\cD. 
\end{equation*}

\noindent
So, $(\ref{LogGrass})$ and $(\ref{ExpSt})$ provide hereafter an explicit expression for the holonomy map of $(\ref{holonomyFormulaSG})$.

\begin{prop}\label{holSG}
Let $P \in \Gr$ and $R \in G_{q} \setminus \Ct(P)$. For $U \in \left( \pi^{SG}_{q} \right)^{-1}(P)$, set 
\begin{equation*}
\Delta := \Log_{P}^{\Gr} \left(R\right) \quad , \quad
V := \begin{pmatrix}
U & \Delta U 
\end{pmatrix} \in \R^{n \times 2q}
\quad , \quad A:=U^{T}\Delta U 
\quad , \quad C:=U^{T}\Delta^{2}U.
\end{equation*}

\noindent
Then, 
\begin{equation}\label{equHolSG}
\cH^{\pi^{SG}_{q}}_{(P, R)}(U) = V \left( \exp_{m} \begin{pmatrix}
A & -C \\
I_{q} & A
\end{pmatrix}
\right) \begin{pmatrix}
I_{q} \\ 0_{q} 
\end{pmatrix} \exp_{m} \left( - A \right),
\end{equation}
\end{prop}

\begin{proof}
By $(\ref{horLiftSG})$, $\Delta^{\sharp}_{U} = \left( \Log_{P_{0}}^{\Gr} \left(R\right) \right) U$. Then, we derive $(\ref{equHolSG})$ by applying $(\ref{ExpSt})$ with $\cD=\Delta^{\sharp}_{U}=\Delta U$. 
\end{proof}

\bigskip

\begin{rem}\label{optimality}
With the notations of $(\ref{holonomyFormulaSG})$, let $\gamma(P,R)$ be the minimal geodesic between $P$ and $R$ in $\Gr$. By definition, $\cH^{\pi^{SG}_{q}}_{(P, R)}(U)$ is the endpoint of the horizontal lift through $U$ of $\gamma(P,R)$. Since $\pi^{SG}_{q}$ is a Riemannian submersion, denoting by $d_{g}^{St}$ the geodesic distance in $\left( \St_{q}, g^{\St} \right)$, \cite[Lemma 26.11]{Michor 2008} implies that 
\begin{equation}\label{minGeoDist}
d_{g}^{St} \left( U, \cH^{\pi^{SG}_{q}}_{(P, R)}(U) \right) = \min \left\{ d_{g}^{St} \left( U, W \right) : W \in \left( \pi^{SG}_{q} \right)^{-1}(R) \right\}.  
\end{equation}

\noindent
Thus, among all $q$-frames generating the range of $R$, $\cH^{\pi^{SG}_{q}}_{(P, R)}(U)$ minimizes the geodesic distance in $\St_{q}$ to $U$.
\end{rem}

\subsubsection{The closed form for a section of $\pii$}

Given a type $\I=(q_{i})_{1 \leq i \leq r}$, any $Q \in O(n)$ is partitioned into $q_{i}$-blocks of its columns: 
\begin{equation*}
Q = \begin{pmatrix}
Q^{(1)} & ... & Q^{(i)} & ... & Q^{(r)} 
\end{pmatrix},  
\end{equation*}

\noindent
where for all $1 \leq i \leq r$, $Q^{(i)} \in \St_{q_{i}}$. For all $i$, the map $\pi^{SG}_{q_{i}} : \St_{q_{i}} \lra G_{q_{i}}$ is denoted by $\pi^{i}$. Then, the maps $(\pi^{i})_{1 \leq i \leq r}$ are related to $\pii$ as follows. 

\begin{lem}\label{linkPiS}
For all $Q \in O(n)$, $\pii(Q) = \left( \pi^{i}\left( Q^{(i)} \right) \right)_{1 \leq i \leq r}$.
\end{lem}
    
\begin{proof}
Let $Q \in O(n)$. Notice that for all $1 \leq i \leq r$, $Q^{(i)}=Q I_{n}^{(i)}$ and $I_{n}^{(i)}\left(I_{n}^{(i)}\right)^{T} = P_{0}^{i}$. So, for any $1 \leq i \leq r$,
\begin{equation*}
\pi^{i}\left( Q^{(i)} \right) = Q^{(i)}\left(Q^{(i)}\right)^{T} = Q I_{n}^{(i)}\left( I_{n}^{(i)}\right)^{T} Q^{T} = Q P_{0}^{i} Q^{T}. 
\end{equation*}

\noindent
Since $\pii(Q) = \left( Q P_{0}^{i} Q^{T} \right)_{1 \leq i \leq r}$, this concludes the proof.
\end{proof}

\noindent\\
For fixed $\cP=(P_{i})_{i} \in \F$, consider the following \textit{open neighborhood} of $\cP$ in $\F$.
\begin{equation}\label{defVPtilde}
\widetilde{\cV}_{\cP} := \left\{ \cR=(R_{i})_{i} \in \F : R_{i} \in G_{q_{i}} \setminus \Ct(P_{i}), 1 \leq i \leq r \right\}. 
\end{equation}

\begin{prop}\label{expLocalSec}
Let $\cP=(P_{i})_{i} \in \F$ and $Q_{\cP} \in \left(\pii \right)^{-1}\left( \cP \right)$. For any $\cR \in \widetilde{\cV}_{\cP}$, let $\fS_{\cP} \left( \cR \right) \in O(n)$ be the orthogonal matrix whose $q_{i}$-block is defined by 
\begin{equation}\label{defLocSecGI}
\left( \fS_{\cP} \left( \cR \right) \right)^{(i)} = 
\cH^{\pi^{i}}_{(P_{i}, R_{i})} \left( Q_{\cP}^{(i)} \right), \quad 1 \leq i \leq r. 
\end{equation}

\noindent
Then, the map $\fS_{\cP} : \widetilde{\cV}_{\cP} \lra O(n)$ is a smooth local section of $\pii$. Furthermore, for all $1 \leq i \leq r$, the explicit  expression of the holonomy map $\cH^{\pi^{i}}_{(P_{i}, R_{i})}$ is provided by Proposition $\ref{holSG}$. 
\end{prop}

\begin{proof}
First, we prove that the map $\fS_{\cP}$ is a section of $\pii$ on $\widetilde{\cV}_{\cP}$. For all $\cR \in \widetilde{\cV}_{\cP}$, 
\begin{equation*}
\pii\left( \fS_{\cP}\left( \cR \right) \right) = \left( \pi^{i}\left( \left( \fS_{\cP} \left( \cR \right) \right)^{(i)} \right) \right)_{i} = \left( \pi^{i}\left( \cH^{\pi^{i}}_{(P_{i}, R_{i})} \left( Q_{\cP}^{(i)} \right) \right) \right)_{i} = \left( R_{i} \right)_{i} = \cR. 
\end{equation*}

\noindent
Indeed, the first equality hereabove follows from Lemma $\ref{linkPiS}$ and the second from $(\ref{defLocSecGI})$. The third one holds because the holonomy map $\cH^{\pi^{i}}_{(P_{i}, R_{i})}$ is valued in the fiber of $R_{i}$ under $\pi^{i}$. Finally, the smoothness of $\mathfrak{S}_{\cP}$ follows from that of the holonomy maps $\left( \cH^{\pi^{i}}_{(P_{i}, R_{i})} \right)_{i}$.
\end{proof}

\subsection{Frames generating families of flags}\label{subsecPertFlags}

Let $\J \subset \R$ be an open interval and $\left\{ \cP(x) : x \in \J \right\}$ a family of flags. For any $x \in \J$, the type of $\cP(x)=(P_{i}(x))_{i}$ is of the form $\I(x) = \left\{ q_{i}(x) : 1 \leq i \leq r(x) \right\}$. We suppose that the following conditions hold, where $M(n)$ denotes the set of all $n \times n$ real matrices. 

\begin{Assum}\label{AssumPertFlags}
$(F1)$ The integer $r(x)$ is a constant $r$ independent of $x \in \J$. 

\smallskip

\noindent
$(F2)$ For all $1 \leq i \leq r$, the map $P_{i}(\cdot) : \J \lra M(n)$ is continuous. 
\end{Assum}

\noindent
By \cite[Lemma I-4.10]{Kato 1995}, $(F2)$ implies that for all $i$, $q_{i}(x)=\rk(P_{i}(x))$ is a constant $q_{i}$ on $\J$. Thus, Assumption $\ref{AssumPertFlags}$ implies that for all $x \in \J$, $\cP(x)$ is of type $\I=(q_{i})_{1 \leq i \leq r}$, i.e. $\cP(x) \in \F$.

\noindent\\
We aim at finding a map $\cQ(\cdot) : \J \lra O(n)$ whose regularity is that of the map $\cP(\cdot) : \J \lra \F$ and such that for all $x \in \J$, the columns of $\cQ(x)$ generate $\cP(x)$ (See Definition $\ref{defiGenerateFlag}$). By Lemma $\ref{FiberPii}$, this latter condition is equivalent to 
\begin{equation}\label{genFlagBis}
\pii \left( \cQ(x) \right) = \cP(x), \quad x \in \J. 
\end{equation}

\noindent
This question is motivated by the eigenvector problem in perturbation theory, detailed in paragraph $\ref{applicationPerTheo}$ below. First, we present the analytic solution of \cite{Kato 1995}, that we only reformulate   in terms of flags. Then, we develop our new geometric solution based on the preceding results of this paper. For both methods, the solution obtained is implicit. However, in both cases, an \textit{explicit} solution is available \textit{locally}, i.e. on a neighborhood of any fixed $x_{0} \in \J$. Finally, we compare all these methods, which reveals that our solution, based on the geometry of flag manifolds, improves the analytic one from \cite{Kato 1995}. 

\smallskip

\noindent
Throughout the sequel, we fix $x_{0} \in \J$ and $Q_{0} \in O(n)$ such that $\pii(Q_{0})=\cP(x_{0})$.

\subsubsection{Analytic solution}

Assume that for all $1 \leq i \leq r$, the map $P_{i}(\cdot) : \J \lra M(n)$ is of class $\cC^{1}$. We denote by $GL(n)$ the general linear group of $\R^{n}$. Then, by solving a \textit{linear differential equation}, it is proved in \cite[Paragraph II-4.5]{Kato 1995} that there exists a map $K(\cdot) : \J \lra GL(n)$ of class $\cC^{1}$ such that for all $x \in \J$,
\begin{equation}\label{equationKato}
K(x)P_{i}(x_{0})K(x)^{-1} = P_{i}(x), \quad x \in \J, \quad 1 \leq i \leq r. 
\end{equation}

\noindent
In fact, by \cite[Paragraph II-6.2]{Kato 1995}, for all $x \in \J$, $K(x) \in O(n)$. So, $(\ref{equationKato})$ means in terms of \textit{flags} that  
\begin{equation}\label{KoperatorKato}
K(x) \cdot \cP(x_{0}) = \cP(x), \quad x \in \J.
\end{equation}

\noindent
Then, consider the map $\cQ_{K}(\cdot)$ defined by 
\begin{equation*}
\cQ_{K}(\cdot) : \J \lra O(n), \quad \cQ_{K}(x) := K(x)Q_{0} \in O(n). 
\end{equation*}

\noindent
 The map $\cQ_{K}(\cdot)$ is of class $\cC^{1}$, because the map $K(\cdot)$ is. Now, by $(\ref{KoperatorKato})$, for all $x \in \J$, 
\begin{equation}\label{genFlagKato}
\pii(\cQ_{K}(x)) = \pii \left( K(x)Q_{0} \right) = K(x) \cdot \pii(Q_{0}) = K(x) \cdot \cP(x_{0}) = \cP(x), 
\end{equation}

\noindent
where the second equality holds because for any $Q \in O(n)$, $\pii(Q):=Q \cdot \cP_{0}$. Therefore, $(\ref{genFlagKato})$ implies that for all $x \in \J$, the columns of $\cQ_{K}(x)$ generate the flag $\cP(x)$, which follows from $(\ref{genFlagBis})$

\begin{rem}
The family of frames $\left\{ \cQ_{K}(x) : x \in \J \right\}$ is implicit. However, one can derive from \cite[Remark II-4.4]{Kato 1995} an explicit map $K(\cdot)$, defined only on a neighborhood $\J_{K}^{0}$ of  $x_{0}$ and valued in $O(n)$, such that $(\ref{KoperatorKato})$ holds. Thus, an explicit solution $\cQ_{K}^{0}(\cdot) : \J_{K}^{0} \lra O(n)$ is available locally.
\end{rem}

\subsubsection{Geometric solution}

Proposition $\ref{propGlobalImplicit}$ hereafter provides a geometric solution, based on the principal bundle structure of the main orbital map $\piio : SO(n) \lra \F$, by applying Lemma $\ref{ppalBundleContract}$ below.

\begin{prop}\label{propGlobalImplicit}
Assume that the map $\cP(\cdot) : \J \lra \F$ is of class $\cC^{1}$. Then, there exists a map $\cQ^{L}(\cdot) : \J \lra SO(n)$ of class $\cC^{1}$ such that for all $x \in \J$, the columns of $\cQ^{L}(x)$ generate the flag $\cP(x)$.
\end{prop}

\begin{proof}
By Lemma $\ref{structureMainOM}$, the main orbital map $\piio : SO(n) \lra \F$ is a principal bundle. Since the interval $\J$ is contractible, Lemma $\ref{ppalBundleContract}$ below implies that there exists a smooth map $\cQ^{L}(\cdot) : \J \lra SO(n)$ such that for all $x \in \J$,
\begin{equation*}
\piio \left( \cQ^{L}(x) \right) = \cP(x). 
\end{equation*}

\noindent
Since $\piio$ is the restriction to $SO(n)$ of $\pii$, the expected conclusion follows from $(\ref{genFlagBis})$. 
\end{proof}

\begin{lem}\label{ppalBundleContract}
Let $\Pi : \E \lra \B$ be a principal bundle, $\M$ a manifold and $F : \M \lra \B$ a map of class $\cC^{1}$. If $\M$ is contractible, then there exists a map $L : \M \lra \E$ of class $\cC^{1}$ such that $\Pi \circ L = F$. 
\end{lem}

\begin{proof}
The \textit{pullback} $F^{*}\Pi$ of $\Pi$ wrt $F$ is also a principal bundle, over $\M$ contractible. So, $F^{*}\Pi$ has a \textit{global section}, from which a map $L$ as in the statement is easily derived. See \cite{Ferrer Garcia and Puerta 1994}, where this argument is developed in detail. 
\end{proof}

\subsubsection{Local explicit frames by geometric approach}

The map $\cQ^{L}(\cdot) : \J \lra SO(n)$ is \textit{implicit}. For fixed $x_{0} \in \J$, we obtain in $(\ref{defQhol})$ below, an \textit{explicit} solution $\cQ_{\cH}(\cdot)$ but defined only on a neighborhood of $x_{0}$. The map $\cQ_{\cH}(\cdot)$ is built from the explicit local section of $\pii$ given in $(\ref{defLocSecGI})$. Namely, let $\J_{\cH}^{0}$ be the \textit{open interval} defined as the connected component of $x_{0}$ in the following set $\J_{\cH}$:  
\begin{equation}\label{condCutPert}
\J_{\cH} := \left\{ x \in \J : \cP(x) \in \widetilde{\cV}_{\cP(x_{0})} \right\}.
\end{equation}

\noindent
Since the map $\cP(\cdot)$ is continuous on $\J$, $\J_{\cH}$ is an \textit{open set}, so that $\J_{\cH}^{0}$ is indeed also. Furthermore, $x \in \J_{\cH}$ iff for all $1 \leq i \leq r$, none of the \textit{principal angles} between $P_{i}(x)$ and $P_{i}(x_{0})$ is equal to $\frac{\pi}{2}$. Consider the following map:  
\begin{equation}\label{defQhol}
\cQ_{\cH}(\cdot) : \J_{\cH}^{0} \lra O(n), \quad \cQ_{\cH}(x) := \fS_{\cP(x_{0})} \left( \cP(x) \right).
\end{equation}

\noindent
Then, the map $\cQ_{\cH}(\cdot)$ is well-defined and $(\ref{defLocSecGI})$ provides an explicit expression of $\cQ_{\cH}(x)$ hereafter:  
\begin{equation}\label{QHxi}
\left( \cQ_{\cH}(x) \right)^{(i)} = \cH^{\pi^{i}}_{(P_{i}(x_{0}), P_{i}(x))} \left( Q_{\cP(x_{0})}^{(i)} \right), \quad 1 \leq i \leq r.
\end{equation}

\noindent
By Propostion $\ref{expLocalSec}$, the map $\fS_{\cP(x_{0})} : \widetilde{\cV}_{\cP(x_{0})} \lra O(n)$ is a smooth  local section of $\pii$. So, on one hand, the map $\cQ_{\cH}(\cdot)$ is \textit{continuous}, since $\cP(\cdot)$ is. On the other hand, for all $x \in \J_{\cH}^{0}$, 
\begin{equation*}
\pii \left( \cQ_{\cH}(x) \right) = \pii \left( \fS_{\cP(x_{0})} \left( \cP(x) \right) \right) = \cP(x). 
\end{equation*}

\noindent
Then, $(\ref{genFlagBis})$ implies that for all $x \in \J_{\cH}^{0}$, the columns of $\cQ_{\cH}(x)$ generate the flag $\cP(x)$.

\subsubsection{Comparison of the solutions}

$(i)$ First, we compare the global implicit solutions, i.e. defined on the whole $\J$. Thus, the geometric solution $\cQ^{L}(\cdot)$ is valued in $SO(n)$, instead of $O(n)$ for $\cQ_{K}(\cdot)$, and is more intrinsic. 

\smallskip

\noindent
$(ii)$ For the local explicit frames $\left\{ \cQ_{\cH}(x) : x \in \J_{\cH}^{0} \right\}$, we only need to suppose that Assumption $\ref{AssumPertFlags}$ holds, while all other solutions require additionally a stronger regularity assumption. 

\smallskip

\noindent
$(iii)$ Finally, the map $\cQ_{\cH}(\cdot)$ is \textit{optimal} in the sense of Lemma $\ref{optimalQH}$ hereafter.

\begin{lem}\label{optimalQH}
For all $x \in \J_{\cH}^{0}$ and $1 \leq i \leq r$, among all $q_{i}$-frames generating the range of $P_{i}(x)$, $\left( \cQ_{\cH}(x) \right)^{(i)}$ minimizes the geodesic distance in $\St_{q_{i}}$ to $Q_{0}^{(i)}$. 
\end{lem}

\begin{proof}
Since $\cQ_{\cH}(x)$ is defined by $(\ref{QHxi})$, we conclude by applying Remark $\ref{optimality}$. 
\end{proof}

\noindent
In particular, for all $1 \leq i \leq r$, denoting by $d_{g}^{i}$ the geodesic distance in $\St_{q_{i}}$, 
\begin{equation*}
d_{g}^{i} \left( \left(\cQ_{\cH}(x)\right)^{(i)} , Q_{0}^{(i)} \right) \leq 
\min \left\{ d_{g}^{i} \left( \left(\cQ_{K}(x)\right)^{(i)} , Q_{0}^{(i)} \right) ;  d_{g}^{i} \left( \left(\cQ^{0}_{K}(x)\right)^{(i)} , Q_{0}^{(i)} \right) \right\}, 
\quad x \in \J_{\cH}^{0} \cap \J_{K}^{0}. 
\end{equation*}

\noindent\\
In conclusion, the map $\cQ_{\cH}(\cdot)$ is the \textit{best solution}, although it is a local one. Indeed, it is explicit, requires only Assumption $\ref{AssumPertFlags}$ and satisfies the optimal property of Lemma $\ref{optimalQH}$. 

\bigskip

\begin{rem}
A local section of the orbital map $\pii : O(n) \lra \F$ based directly on the holonomy maps between fibers under $\pii$ would provide a more accurate local geometric solution. However, an explicit expression for such holonomy maps requires a closed form for the Logarithm map in $\F$, which is not known yet. 
\end{rem}

\subsection{Application to perturbation of eigenvectors}\label{applicationPerTheo}

In the sequel, we fix $\cS_{0} \in \Sy_{n}$. A perturbation of $\cS_{0}$ is a map $\cS(\cdot) : \J \lra \Sy_{n}$, where $\J \subset \R$ is an open interval containing $x_{0}$ such that $\cS(x_{0})=\cS_{0}$. A perturbation of eigenvectors of $\cS_{0}$ is a map $\cQ^{\cS}(\cdot) : \J \lra O(n)$ such that for all $x \in \J$, $\cQ^{\cS}(x)$ is a matrix of eigenvectors of $\cS(x)$. In perturbation theory, the question of the existence of a map $\cQ^{\cS}(\cdot)$ with some regularity, inherited from that of  $\cS(\cdot)$, is called the eigenvector problem.

\subsubsection{Flag of eigenspaces of a symmetric matrix}

For $\cS \in \Sy_{n}$, denote by $r(\cS)$ the number of distinct eigenvalues of $\cS$ and by $\Lambda(\cS)$ its spectrum: 
\begin{equation*}
\Lambda(\cS) = \left\{ \lambda_{1}(\cS) > ... > \lambda_{i}(\cS) > ... > \lambda_{r(\cS)}(\cS) \right\}, 
\end{equation*}

\noindent
where, for all $1 \leq i \leq r(\cS)$, $\lambda_{i}(\cS)$ is an eigenvalue of $\cS$. For all $i$, let $P_{i}(\cS)$ be the \textit{eigenprojection} associated to $\lambda_{i}(\cS)$, i.e. the orthogonal projection onto the eigenspace associated to $\lambda_{i}(\cS)$. Set
\begin{equation*}
\cP^{\cS} := \left( P_{i}(\cS) \right)_{1 \leq i \leq r(\cS)}. 
\end{equation*}

\noindent
Then, the spectral theorem implies readily that $\cP^{\cS}$ is a flag, called the \textit{flag of eigenspaces} of $\cS$. 

\begin{lem}\label{flagEigenPii}
Let $\cS \in \Sy_{n}$ and $\I=(q_{i})_{1 \leq i \leq r}$ the type of its flag of eigenspaces $\cP^{\cS}$. Let $Q \in O(n)$. Then, $Q$ is a matrix of eigenvectors of $\cS$ iff $\pii(Q)=\cP^{\cS}$. 
\end{lem}

\begin{proof}
Clearly, $Q$ is a matrix of eigenvectors of $\cS$ iff its columns generate the flag $\cP^{\cS}$, which is equivalent to $\pii(Q)=\cP^{\cS}$, according to Lemma $\ref{FiberPii}$.
\end{proof}

\subsubsection{Geometric solution to the eigenvector problem}

In the sequel, we suppose that the following conditions hold. 

\begin{Assum}\label{AssumPertSym}
$(S1)$ The map $\cS(\cdot) : \J \lra \Sy_{n}$ is analytic. 

\smallskip

\noindent
$(S2)$ The number $r(x)$ of distinct eigenvalues of $\cS(x)$ is a constant $r$ on $\J$.
\end{Assum}

\noindent
In fact, $(S1)$ implies that $r(x)$ is constant on $\J$, except for exceptional values of $x$: See p.64 in \cite{Kato 1995}. Thus, $(S2)$ means that $\J$ contains no such point, which is supposed most of the time in \cite{Kato 1995}.

\noindent\\
For all $x \in \J$, the flag of eigenspaces of $\cS(x)$ is denoted by $\cP^{\cS}(x)=\left( P_{i}^{\cS}(x) \right)_{i}$. Obviously, $(S2)$ means that $(F1)$ holds for the family of flags $\left\{ \cP^{\cS}(x) : x \in \J  \right\}$. Now, by \cite[Theorem II-6.1]{Kato 1995}, $(S1)$ implies that for all $1 \leq i \leq r$, the eigenprojection map $P_{i}^{\cS}(\cdot) : \J \lra \Sy_{n}$ is also analytic. So, $(F2)$ of Assumption $\ref{AssumPertFlags}$ holds for $\left( \cP^{\cS}(x) \right)_{x \in \J}$, to which we can therefore apply the results of subsection $\ref{subsecPertFlags}$. In particular, the flags $\left( \cP^{\cS}(x) \right)_{x \in \J}$ have a \textit{constant type} $\I=(q_{i})_{1 \leq i \leq r}$, i.e. for all $x \in \J$, $\cP^{\cS}(x) \in \F$.

\begin{rem}
The analyticity of the map $\cS(\cdot)$ seems to be a too strong assumption to insure that the eigenprojections are continuous. However, Example 5.3 p.111 in \cite{Kato 1995} provides a case of a map $T(\cdot) : \R \lra \Sy_{2}$ that is even smooth, but for which the eigenprojections are not continuous. 
\end{rem}

\noindent
By Lemma $\ref{flagEigenPii}$, a solution to the eigenvector problem is a family of frames generating the flags $\left( \cP^{\cS}(x) \right)_{x \in \J}$. Now, according to subsection $\ref{subsecPertFlags}$, the best family of such frames is that given by $(\ref{defQhol})$, where $\cP(\cdot)$ is replaced by $\cP^{\cS}(\cdot)$. So, these frames define a map $\cQ_{\cH}^{\cS}(\cdot)$ valued in $O(n)$:  
\begin{equation}\label{defQholBis}
\cQ_{\cH}^{\cS}(\cdot) := \fS_{\cP^{\cS}(x_{0})}(\cdot) \circ \cP^{\cS}(\cdot)
\end{equation}

\noindent
So, the map $\cQ_{\cH}^{\cS}(\cdot)$ is a \textit{smooth} and \textit{explicit} local (i.e. not defined on the whole $\J$) perturbation of eigenvectors of $\cS_{0}$, for which the optimal property of Lemma $\ref{optimalQH}$ holds, improving that given in \cite{Kato 1995}.

\section*{Acknowledgements}
The authors have received funding from the European Research Council (ERC) under the European Union’s Horizon 2020 research and innovation program (grant agreement G-Statistics No 786854). It was also supported by the French government through the 3IA Côte d’Azur Investments ANR-19-P3IA-0002 managed by the National Research Agency. The authors thank Yann Thanwerdas for his careful proofreading of this manuscript.

\section{Appendix}\label{sect5}

\subsection{Proofs of Section $\ref{sect3}$}

\subsubsection{Proof of Lemma $\ref{UisOpen}$}

\begin{proof}
Let $\Gras(K, \g)$ be the Grassmannian of all linear subspaces of dimension $K$ in $\g$. Consider the map 
\begin{equation*}
\kappa : \B \lra \Gras(K, \g), \quad \kappa(b) := \Ad(Q_{b})(\fk)=\fk_{b}, 
\end{equation*}

\noindent
which follows from $(ii)$ of Lemma $\ref{MBandMH}$. Now, $\fk_{b_0}=\fk$, so that $b_{0} \in \U$ and 
\begin{equation*}
\U = \kappa^{-1}\left( \mathfrak{U}_{\fm} \right)
\quad \textrm{where} \quad  
\mathfrak{U}_{\fm} := \left\{ \fv \in \Gras(K, \g) : \fv \cap \fm = \left\{0\right\} \right\}. 
\end{equation*}
 
\noindent
It is well-known that $\mathfrak{U}_{\fm}$ is an open subset of $\Gras(K, \g)$. So, it is enough to prove that $\kappa$ is continuous. Now

 \[ \kappa : \xymatrix{
  b \ar@{|->}[r]^{\sigma_{0}} & Q_{b} \ar@{|->}[r]^{\Ad} & \Ad(Q_{b}) \ar@{|->}[r]^{E_{\fk}} & \Ad(Q_{b})(\fk) 
  }, \qquad \textrm{i.e.} \qquad \kappa = E_{\fk} \circ \Ad \circ \sigma_{0}, 
\]

\noindent
where $E_{\fk} : \Aut(\g) \lra \Gras(K, \g)$ is defined by $F \lmp F(\fk)$. First, $\sigma_{0}$ and $\Ad$ are smooth. Secondly, Proposition $\ref{biInvScal}$ implies that $\Ad$ is valued in the orthogonal group $O(\g)$ of $\left( \g, \left\langle \cdot, \cdot \right\rangle \right)$. Thirdly, when $\Gras(K, \g)$ is identified with orthogonal projectors, the restriction $\left(E_{k}\right)_{|O(\g)} : O(\g) \lra \Gras(K, \g)$ is of the form $\widetilde{\mathcal{Q}} \lmp \widetilde{\mathcal{Q}}P_{\fk} \widetilde{\mathcal{Q}}^{T}$, where $P_{\fk}$ is the projector associated to $\fk$. So, $\left(E_{k}\right)_{|O(\g)}$ is smooth. Therefore, $\kappa$ is smooth. 
\end{proof}

\subsubsection{Proof of Proposition $\ref{PropDYbar}$}

\begin{proof}
By Lemma $\ref{inverseProjection}$,
\begin{equation*}
\overline{Y_{\beta}}(b) = p_{b}^{-1}\left( V_{\beta}(b) \right) = \sum\limits_{k=1}^{N} \left( \sum\limits_{\ell=1}^{N} v_{k,\ell}(b) \right) \epsilon_{k}
\quad \textrm{where} \quad
v_{k,\ell}(b) = \left( \cA(b)^{-1} \right)_{k, \ell} \left\langle V_{\beta}(b), \epsilon'_{\ell}(b) \right\rangle. 
\end{equation*}

\noindent
So, for any $\Delta \in T_{b_{0}}\B$, 
\begin{equation}\label{AppDiffYbeta}
\left( d_{b_{0}}\overline{Y_{\beta}} \right)(\Delta) = \sum\limits_{k=1}^{N} \left( \sum\limits_{\ell=1}^{N} \left( d_{b_{0}} v_{k,\ell} \right)(\Delta) \right) \epsilon_{k} := 
\sum\limits_{k=1}^{N} \left( \sum\limits_{\ell=1}^{N} w^{1}_{k,\ell}(\Delta) + w^{2}_{k,\ell}(\Delta) + w^{3}_{k,\ell}(\Delta) \right) \epsilon_{k}
\end{equation}

\noindent
where, for any $1 \leq k \leq \ell \leq N$,  
\begin{equation}\label{w1}
w^{1}_{k,\ell}(\Delta) = \left( d_{b_{0}} \left( \cA(\cdot)^{-1} \right)_{k, \ell} \right)(\Delta) \left\langle V_{\beta}(b_{0}), \epsilon'_{\ell}(b_{0}) \right\rangle, 
\end{equation}

\begin{equation}\label{w2}
w^{2}_{k,\ell}(\Delta) = \left( \cA(b_{0})^{-1} \right)_{k, \ell} \left\langle \left( d_{b_{0}} V_{\beta} \right)(\Delta), \epsilon'_{\ell}(b_{0}) \right\rangle, 
\end{equation}

\begin{equation}\label{w3}
w^{3}_{k,\ell}(\Delta) = \left( \cA(b_{0})^{-1} \right)_{k, \ell} \left\langle V_{\beta}(b_{0}), \left( d_{b_{0}} \epsilon'_{\ell} \right)(\Delta) \right\rangle.
\end{equation}

\noindent\\
Now, recall that for all $b \in \V$, $\epsilon'_{\ell}(b) = \Ad(Q_{b})(\epsilon_{\ell}) = \Ad(\sigma(b))(\epsilon_{\ell})$. Thus, $\epsilon'_{\ell}(b_{0})=\epsilon_{\ell}$. Furthermore $\sigma(b_{0})=e$, so that $d_{\sigma(b_{0})}\Ad=\ad$. Therefore, by $(\ref{adBracket})$, which relates the adjoint representation on $\g$ to the Lie bracket on $\g$, 
\begin{equation}\label{diffEpsilonPrime}
\left( d_{b_{0}} \epsilon'_{\ell} \right)(\Delta) = \ad\left( \left(d_{b_{0}}\sigma\right)(\Delta) \right)(\epsilon_{\ell}) = \lb \left(d_{b_{0}}\sigma\right)(\Delta) , \epsilon_{\ell} \rb.
\end{equation}

\noindent
Then, recall that for all $b \in \V$, $\cA(b)_{k, \ell} = \left\langle \epsilon_{\ell}, \epsilon'_{k}(b) \right\rangle$. Thus, $\cA(b_{0})_{k, \ell} = \left\langle \epsilon_{\ell}, \epsilon'_{k}(b_{0}) \right\rangle = \left\langle \epsilon_{\ell}, \epsilon_{k} \right\rangle = \delta_{k,\ell}$, where $\delta_{k,\ell}$ is the Kronecker symbol, and $\cA(b_{0})=I_{N}$. So, the derivative of the inverse of a matrix provides that    
\begin{align*}
\left( d_{b_{0}} \left( \cA(\cdot)^{-1} \right)_{k, \ell} \right)(\Delta) &= - \left\{ \cA(b_{0})^{-1} \left( d_{b_{0}} \cA \right)(\Delta) \cA(b_{0})^{-1} \right\}_{k, \ell} \\
&= - \left(d_{b_{0}}\cA_{k, \ell}\right) (\Delta) \\
&= - \left\langle \epsilon_{\ell}, \left(d_{b_{0}}\epsilon'_{k}\right) (\Delta) \right\rangle \\
&= - \left\langle \epsilon_{\ell}, \lb \left(d_{b_{0}}\sigma\right)(\Delta) , \epsilon_{k} \rb \right\rangle \\ 
&= - \left\langle \lb \epsilon_{\ell}, \left(d_{b_{0}}\sigma\right)(\Delta)  \rb , \epsilon_{k} \right\rangle,
\end{align*}

\noindent
where the fourth equality follows from $(\ref{diffEpsilonPrime})$ and the fifth from $(\ref{AdInvariance})$. By $(\ref{w1})$, we deduce that 
\begin{align*}
\sum\limits_{k=1}^{N} \left( \sum\limits_{\ell=1}^{N} w^{1}_{k,\ell}(\Delta) \right) \epsilon_{k} &=
-  \sum\limits_{k=1}^{N} \left( \sum\limits_{\ell=1}^{N} \left\langle \lb \epsilon_{\ell}, \left(d_{b_{0}}\sigma\right)(\Delta)  \rb , \epsilon_{k} \right\rangle \left\langle V_{\beta}(b_{0}), \epsilon_{\ell} \right\rangle \right) \epsilon_{k} \\
&= - \sum\limits_{\ell=1}^{N} \left\langle V_{\beta}(b_{0}), \epsilon_{\ell} \right\rangle \left( \sum\limits_{k=1}^{N} \left\langle \lb \epsilon_{\ell}, \left(d_{b_{0}}\sigma\right)(\Delta) \rb , \epsilon_{k} \right\rangle \epsilon_{k} \right) \\ 
&= - \sum\limits_{\ell=1}^{N} \left\langle V_{\beta}(b_{0}), \epsilon_{\ell} \right\rangle \lb \epsilon_{\ell}, \left(d_{b_{0}}\sigma\right)(\Delta) \rb \\
&= - \lb V_{\beta}(b_{0}), \left(d_{b_{0}}\sigma\right)(\Delta) \rb.
\end{align*}

\noindent
Then, by $(\ref{w2})$, 
\begin{align*}
\sum\limits_{k=1}^{N} \left( \sum\limits_{\ell=1}^{N} w^{2}_{k,\ell}(\Delta) \right) \epsilon_{k} &= 
\sum\limits_{k=1}^{N} \left( \sum\limits_{\ell=1}^{N} \delta_{k,\ell} \left\langle \left( d_{b_{0}} V_{\beta} \right)(\Delta), \epsilon_{\ell} \right\rangle \right) \epsilon_{k} \\
&= \sum\limits_{k=1}^{N} \left\langle \left( d_{b_{0}} V_{\beta} \right)(\Delta), \epsilon_{k} \right\rangle \epsilon_{k} \\
&= \left( d_{b_{0}} V_{\beta} \right)(\Delta). 
\end{align*}

\noindent
Finally, by $(\ref{w3})$ and $(\ref{AdInvariance})$, 
\begin{align*}
\sum\limits_{k=1}^{N} \left( \sum\limits_{\ell=1}^{N} w^{3}_{k,\ell}(\Delta) \right) \epsilon_{k} &=
\sum\limits_{k=1}^{N} \left( \sum\limits_{\ell=1}^{N} \delta_{k, \ell} \left\langle V_{\beta}(b_{0}) , \lb \left(d_{b_{0}}\sigma\right)(\Delta) , \epsilon_{\ell} \rb \right\rangle \right) \epsilon_{k} \\
&= \sum\limits_{k=1}^{N} \left\langle \lb V_{\beta}(b_{0}) , \left(d_{b_{0}}\sigma\right)(\Delta) \rb , \epsilon_{k} \right\rangle \epsilon_{k} \\
&= \lb V_{\beta}(b_{0}) , \left(d_{b_{0}}\sigma\right)(\Delta) \rb \\
&= - \sum\limits_{k=1}^{N} \left( \sum\limits_{\ell=1}^{N} w^{1}_{k,\ell}(\Delta) \right) \epsilon_{k}. 
\end{align*}

\noindent
Therefore, we deduce from the preceding equations and $(\ref{AppDiffYbeta})$ that
\begin{equation*}
\left( d_{b_{0}}\overline{Y_{\beta}} \right)(\Delta) = \sum\limits_{k=1}^{N} \left( \sum\limits_{\ell=1}^{N} w^{2}_{k,\ell}(\Delta) \right) \epsilon_{k} = \left( d_{b_{0}} V_{\beta} \right)(\Delta). 
\end{equation*}

\noindent
This concludes the proof of Proposition $\ref{PropDYbar}$. 
\end{proof}

\subsection{Proofs of section $\ref{sect4}$}

\subsubsection{Proof of Lemma $\ref{FiberPii}$}

\begin{proof}
For any $1 \leq i \leq r$, let $p_{i} : \R^{n} \lra \R^{n}$ be the linear map whose matrix in the canonical basis of $\R^{n}$ is $P_{i}$. Thus, $p_{i}$ is the orthogonal projection onto $\Img(P_{i})$. Now, $Q^{-1}P_{i}Q$ is the matrix of $p_{i}$ in the basis of $\R^{n}$ formed by the columns of $Q$. Consequently: for all $1 \leq i \leq r$, the columns of $Q^{(i)}$ form an orthonormal basis of $\Img(P_{i})$ if and only if, for all $1 \leq i \leq r$, $Q^{-1}P_{i}Q=P^{i}_{0}$ i.e. $QP^{i}_{0}Q^{T}=P_{i}$.
\end{proof}

\subsubsection{Proof of Lemma \ref{PrelimHorLift}}

First, we introduce the following notations. A type $\I=\left( q_{i} \right)_{1 \leq i \leq r}$ induces a partition of $(1, ..., n)$ into blocks of indices as follows. 
\begin{equation*}
B_{1} = (1, ..., q_{1}) ,~ ... ~, B_{i} = (q_{i-1}+1, ..., q_{i-1}+q_{i}) ,~ ... ~, B_{r} = (q_{r-1}+1, ..., q_{r}), \quad 1 < i < r.
\end{equation*}

\noindent
Let $M=\left(m_{\alpha,\beta}\right)_{1 \leq \alpha \leq \beta \leq n}$ be a $n \times n$ matrix. For $1 \leq i \leq j \leq n$, we introduce the following submatrix of $M$.
\begin{equation}\label{blocksWrtI}
M^{(i,j)} := \left(m_{\alpha,\beta}\right)_{\alpha \in B_{i}, \beta \in B_{j}}.
\end{equation}

\noindent
Now, we come to the proof itself. 

\begin{proof}
For any $1 \leq i \leq r$, set 
\begin{equation*}
\Omega[i] := \lb \lb \Omega, P_{0}^{i} \rb, P_{0}^{i} \rb.
\end{equation*}

\noindent
After computations, we obtain that $\Omega[i]$ is the skew-symmetric matrix such that for all $j \neq i$, $\left(\Omega[i] \right)^{(i,j)}=\Omega^{(i,j)}$ and $\left( \Omega[i] \right)^{(i,i)}=0$, while all other blocks of $\Omega[i]$ are zero matrices. So, in $\sum\limits_{i=1}^{r} \Omega[i]$, for any $i_{0} \neq j_{0}$, the block $\Omega^{(i_{0},j_{0})}$ is counted twice: once from the $i_{0}$-th block of lines of $\Omega[i_{0}]$ and once from the $j_{0}$-th block of columns of $\Omega[j_{0}]$. By this, we mean that, 
\begin{equation*}
\left( \sum\limits_{i=1}^{r} \Omega[i] \right)^{(i_{0},j_{0})} = 2 \Omega^{(i_{0},j_{0})}, \quad i_{0} \neq j_{0}.
\end{equation*}

\noindent
On the other hand, $\left( \sum\limits_{i=1}^{r} \Omega[i] \right)^{(i_{0},i_{0})}=0$. Therefore, $\sum\limits_{i=1}^{r} \Omega[i] =2\Omega^{h}$. 
\end{proof}

\subsubsection{Proof of Lemma $\ref{piSG}$}

\begin{proof}
We deduce from results of subsection 2.2 in \cite{Bendokat Zimmermann and Absil 2020} that the horizontal spaces are the same for $g^{\St}$ and the canonical metric on $\St_{q}$. Now, for the latter metric, $(\ref{horLiftSG})$ holds. So $(\ref{horLiftSG})$ also holds for $g^{\St}$. This proves $(i)$. For $(ii)$, we notice first that, by results of \cite{Bendokat Zimmermann and Absil 2020}, $\Pi^{SG}_{q}$ is a smooth submersion. Then, for $P \in \Gr$, and $\Delta, \widetilde{\Delta} \in T_{P}\Gr$, 
\begin{equation*}
g^{\Gr}_{P}\left(\Delta, \widetilde{\Delta} \right) = \mathrm{tr}\left( \left(\Delta_{U}^{\sharp}\right)^{T} \widetilde{\Delta}_{U}^{\sharp} \right) = g^{\St}_{U}\left( \Delta_{U}^{\sharp}, \widetilde{\Delta}_{U}^{\sharp} \right), \quad 
 U \in \left( \pi^{SG}_{q} \right)^{-1}\left(P\right). 
\end{equation*}

\noindent
This follows from $(3.2)$ in \cite{Bendokat Zimmermann and Absil 2020} and $(i)$. Thus, $(ii)$ is proved.
\end{proof}


\bibliographystyle{plain}

\end{document}